\documentclass{amsart}

\usepackage{amssymb} \usepackage{amsfonts} \usepackage{amsmath}
\usepackage{amsthm} \usepackage{epsfig} 
\usepackage{color}
\usepackage{pinlabel}
\usepackage{amscd}
\usepackage[all]{xy}
\usepackage{pinlabel}
\usepackage{tcolorbox}
\usepackage{seqsplit}

\newcommand{\hash}[1]{{\ttfamily\seqsplit{#1}}}

\newtheorem{lemma}{Lemma}
\newtheorem{teo}[lemma]{Theorem}
\newtheorem{prop}[lemma]{Proposition}
\newtheorem{cor}[lemma]{Corollary}

\newtheorem{guess}[lemma]{Guess}

\theoremstyle{definition}
\newtheorem{defn}[lemma]{Definition}

\newtheorem{quest}[lemma]{Question}
\newtheorem{probl}[lemma]{Problem}

\theoremstyle{remark}
\newtheorem{rem}[lemma]{Remark}

\newcommand{\Out}{{\rm Out}}

\newcommand{\matR} {\ensuremath {\mathbb{R}}}

\newcommand{\matZ} {\ensuremath {\mathbb{Z}}}

\newcommand{\matH} {\ensuremath {\mathbb{H}}}

\newcommand{\calS} {\ensuremath {\mathcal{S}}}

\newcommand{\id} {\ensuremath {{\rm id}}}

\newcommand{\Isom} {\ensuremath {{\rm Isom}}}
\newcommand{\lk} {\ensuremath {{\rm lk}}}
\newcommand{\St} {\ensuremath {{\rm St}}}
\newcommand{\HW} {\ensuremath {{\rm HW}}}

\newcommand{\nota} [1] {\caption{\footnotesize{#1}}}

\author{Giovanni Italiano}
\address{Scuola Normale Superiore, Piazza dei Cavalieri, 7, 56126 Pisa, Italy}
\email{giovanni dot italiano at sns dot it}

\author{Bruno Martelli}
\address{Dipartimento di Matematica, Largo Pontecorvo 5, 56127 Pisa, Italy}
\email{bruno dot martelli at unipi dot it}

\author{Matteo Migliorini}
\address{Scuola Normale Superiore, Piazza dei Cavalieri, 7, 56126 Pisa, Italy}
\email{matteo dot migliorini at sns dot it}

\title{Hyperbolic 5-manifolds that fiber over $S^1$}

\begin{document}

\begin{abstract}
We exhibit some finite-volume cusped hyperbolic 5-manifolds that fiber over the circle. These include the smallest hyperbolic 5-manifold known, discovered by Ratcliffe and Tschantz. As a consequence, we build a finite type subgroup of a hyperbolic group that is not hyperbolic.
\end{abstract}

\maketitle

\section*{Introduction} \label{introduction:section}
In 1977 J\o rgensen exhibited the first examples of hyperbolic 3-manifolds that fiber over the circle \cite{J} and opened up a subject that has been of central importance in low-dimensional topology in the last 40 years, with the celebrated contributions of Thurston \cite{Th:s, Th:II} and more recently of Agol and Wise \cite{A, W}. Since then, it has been an open question whether such fibrations may occur in higher dimensions. We provide here an affirmative answer.

We work in the smooth category, so fibrations are smooth. Every hyperbolic manifold $M$ in this introduction is tacitly assumed to be finite-volume and complete, unless otherwise mentioned. When $M$ is not compact, it is the interior of a compact manifold $\overline M$ with boundary (we say as usual that $M$ is \emph{cusped}) and a fibration for $M$ is by definition the restriction of a fibration of $\overline M$.

An even-dimensional hyperbolic manifold cannot fiber over the circle because of the Chern -- Gauss -- Bonnet Theorem. The first dimension to look at is then $n=5$. We prove here the following.

\begin{teo} \label{main:teo}
There are some cusped hyperbolic 5-manifolds that fiber over $S^1$.
\end{teo}

We show in particular that the smallest hyperbolic 5-manifold known, constructed by Ratcliffe and Tschantz in \cite{RT5}, fibers over the circle, and the fiber is a 4-manifold with $\chi = 1$. This is somehow analogous to the figure-eight knot complement, a famous hyperbolic 3-manifold of small volume that fibers over the circle with fiber a punctured torus with $\chi = -1$.

As noted by Thurston \cite{Th:II}, the existence of hyperbolic manifolds that fiber over the circle is quite paradoxical, because the fiber is very far from being geodesic: its counterimage in $\matH^n$ consists of disjoint smooth hyperplanes, each placed at uniformly bounded distance from each other and having the whole sphere at infinity as a limit set.

In higher dimension $n\geq 5$ the situation is even more paradoxical because the monodromy of the fiber $F$ has infinite order in $\Out(\pi_1(F))$, so in particular $F$ cannot admit any hyperbolic metric by Mostow Rigidity. By pursuing along this line we deduce the following. A group $G$ is of \emph{finite type} if it is the fundamental group of a finite aspherical cell complex.

\begin{cor} \label{main:cor}
There is a hyperbolic group $G$ that contains a subgroup $H<G$ of finite type that is not hyperbolic.
\end{cor}

The existence of such a pair $H< G$ was a well-known open problem, raised in particular by Bestvina \cite[Question 1.1]{B}, Brady \cite[Question 7.2]{Bra}, Bridson \cite[Question 4.1]{Br}, and Jankiewicz, Norin, and Wise \cite[Section 7]{JNW}. A finitely presented subgroup of a hyperbolic group that is not hyperbolic was first exhibited by Brady in \cite{Bra}. On the other hand, every finite type subgroup of a hyperbolic group of cohomological dimension 2 is hyperbolic \cite{Ge}.

The group $H$ constructed here is the kernel of a surjective homomorphism $G \to \matZ$, and it is the fundamental group of a 4-dimensional finite aspherical complex obtained by coning all the 3-tori boundaries of some compact aspherical 4-manifold. The homomorphism $G \to \matZ$ is built by taking a fibering cusped hyperbolic 5-manifold with large cusp sections and coning the boundary components of the fibers, see Section \ref{GGT:section}. The cohomological dimensions of $H$ and $G$ are 4 and 5.

\begin{cor}
There is a finite type group $H$ that is not hyperbolic and does not contain any Baumslag -- Solitar subgroup $BS(m,n)$.
\end{cor}
\begin{proof}
The group $H$ is contained in a hyperbolic group $G$, and hyperbolic groups contain no Baumslag -- Solitar subgroups.
\end{proof}

This answers a well-known question, raised in particular by Bestvina \cite{B}, Bridson \cite[Question 2.22]{Br} and Drutu and Kapovich \cite[Problem 11.129]{DK}.
We note that there are some classes of groups for which it is known that they are hyperbolic if and only if they do not contain any Baumslag -- Solitar subgroup: these include the free-by-cyclic groups \cite{Bri} and more generally the ascending HNN extensions of free groups \cite{Mu},

Turning back to our fibering hyperbolic 5-manifold, by taking the abelian cover determined by the fibration we deduce the following.

\begin{cor}
There are some geometrically infinite hyperbolic 5-manifolds that are diffeomorphic to $F\times \matR$, for some 4-manifold $F$ that is diffeomorphic to the interior of a compact 4-manifold with boundary.
\end{cor}

Finally, by taking the finite cyclic covers over a fibering 5-manifold we deduce immediately the following.

\begin{cor}
There are infinitely many hyperbolic 5-manifolds with bounded Betti numbers $\sum_{i=0}^5 b_i < K$, for some $K>0$.
\end{cor}

\begin{proof}
Let $M \to S^1$ be a fibering hyperbolic 5-manifold with fiber $F$. Let $M_{n} \to M$ be the degree-$n$ cyclic covering determined by the kernel of $\pi_1(M) \to \pi_1(S^1) = \matZ \to \matZ/{n\matZ}$. The manifold $M_n$ also fibers with fiber $F$. Using Mayer-Vietoris we easily see that $b_i(M_n)$ is bounded above by some number that depends only on $F$.
\end{proof}

\subsection*{Structure of the paper}
In Section \ref{fibration:section} we construct a fibration $f\colon M^5 \to S^1$ on the cusped hyperbolic 5-manifold $M^5$ that was previously considered in \cite{IMM}. This manifold has 40 cusps, and the main tool that we use to build $f$ is the picewise-linear Morse theory of Bestvina and Brady \cite{BB}.

In Section \ref{quotient:section} we manipulate this example to construct a fibration on the much smaller Ratcliffe -- Tschantz hyperbolic 5-manifold $N^5$, which has 2 cusps and is in fact commensurable with $M^5$. 
The fiber of this fibration is an aspherical 4-manifold $F^4$ with $\chi(F^4) = 1$ that we can study and describe explicitly using Regina \cite{regina}.

We prove Corollary \ref{main:cor} in Section \ref{GGT:section}. Given Theorem \ref{main:teo}, the proof of the corollary is relatively short and contains some standard arguments that are known to experts: the only technical difficulties arise from the fact that $M^5$ is cusped, but these are overcome by applying the Dehn filling theorems of Fujiwara and Manning \cite{FM}. An effect of this is that the non-hyperbolic group $H$ is the fundamental group of an aspherical 4-dimensional \emph{pseudo}-manifold (which is not a manifold).

In Section \ref{further:section} we make some general comments and suggest some future directions of research.

\subsection*{Acknowledgements} We thank Roberto Frigerio and Giles Gardam for valuable discussions.

\section{The fibration} \label{fibration:section}
We prove here Theorem \ref{main:teo}, that is we provide some examples of cusped hyperbolic 5-manifolds that fiber over the circle. The main protagonist is the manifold $M^5$, that we have already encountered in \cite{IMM}. This manifold is constructed from a right-angled polytope $P^5$ via a standard colouring technique that we briefly recall below.

To build the fibration $f\colon M^5 \to S^1$, we work in the piecewise-linear category, and start by constructing a piecewise-linear $f$. The main tool is the picewise-linear Morse theory of Bestvina and Brady \cite{BB}, and we are heavily inspired by the more recent paper of Jankiewicz, Norin, and Wise \cite{JNW}. We give $M^5$ an affine cell complex structure and then define $f$ by glueing altogether some height functions on each affine cell. Since the ascending and descending links at all vertices collapse to points, the function $f$ is smoothable to a fibration.

In Section \ref{quotient:section} we will quotient the manifold $M^5$ by some isometries to recover some smaller example where the fiber of the fibration is combinatorially simpler. This smaller example will turn out \emph{a posteriori} to be the manifold exhibited by Ratcliffe and Tschantz \cite{RT5}.

\subsection{The polytope $P^5$}
We describe the 5-dimensional hyperbolic right-angled polytope $P^5\subset \matH^5$, already considered by various authors \cite{ALR, IMM, PV, RT5}. We use the Klein model for $\matH^5$, embedded in $\matR^5$ as the unit ball.

The polytope $P^5$ is the intersection of the 16 half-spaces 
$$\pm x_1 \pm x_2 \pm x_3 \pm x_4 \pm x_5 \leq 1$$
having an even number of minus signs. It has 16 facets and 26 vertices. 
The 16 facets are orthogonal to the vectors
$$(\pm 1, \pm 1, \pm 1, \pm 1, \pm 1)$$
with an even number of minus signs. The 26 vertices consist of the 10 ideal vertices
$$(\pm 1,0,0,0,0),\quad \ldots, \quad (0,0,0,0,\pm 1) $$
and of the 16 real vertices
$$\frac 13 (\pm 1, \pm 1, \pm 1, \pm 1, \pm 1)$$
with an odd number of minus signs. Every facet of $P^5$ is opposed to a real vertex.
The isometry group $\Isom(P^5)$ has order $2^4\cdot 5! = 1920$ and consists of all the maps
$$(x_1,x_2,x_3,x_4,x_5) \longmapsto (\pm x_{\sigma(1)}, \pm x_{\sigma(2)}, \pm x_{\sigma(3)}, \pm x_{\sigma(4)}, \pm x_{\sigma(5)})$$
where $\sigma \in S_5$ is any permutation and there is an even number of minus signs.

The polytope $P^5$ is right-angled. Two facets of $P^5$ are adjacent if the corresponding orthogonal vectors differ in only two signs. Every facet of $P^5$ is isometric to the polytope $P^4$ considered in \cite{RT4}. The isometry group of $P^5$ acts transitively on its facets. The adjacency graph of the facets of $P^5$ is shown in Figure \ref{P5:fig}. The polytope $P^4$ has 10 facets, so every vertex in the adjacency graph is connected to 10 other vertices.

Every 3-dimensional face of $P^5$ is isometric to the right-angled double pyramid $P^3$ with triangular base, with 3 ideal vertices and 2 real ones, see \cite{RT4}. Every 2-dimensional face of $P^5$ is a triangle with two ideal vertices and one (right-angled) real one.

\begin{figure}
\labellist
\small\hair 2pt
\pinlabel $+++++$ at 260 535
\pinlabel $----+$ at 260 -10
\pinlabel $---+-$ at 65 60
\pinlabel $--+--$ at 465 60
\pinlabel $-+---$ at 160 170
\pinlabel $+----$ at 370 170
\pinlabel $--+++$ at 255 140

\pinlabel $++-+-$ at 65 463
\pinlabel $+++--$ at 465 463
\pinlabel $-+++-$ at 160 350
\pinlabel $+-++-$ at 370 350
\pinlabel $++--+$ at 255 390

\pinlabel $+-+-+$ at 505 240
\pinlabel $-++-+$ at 405 280
\pinlabel $+--++$ at 155 280
\pinlabel $-+-++$ at 5 240

\endlabellist
 \begin{center}
  \includegraphics[width = 12.5 cm]{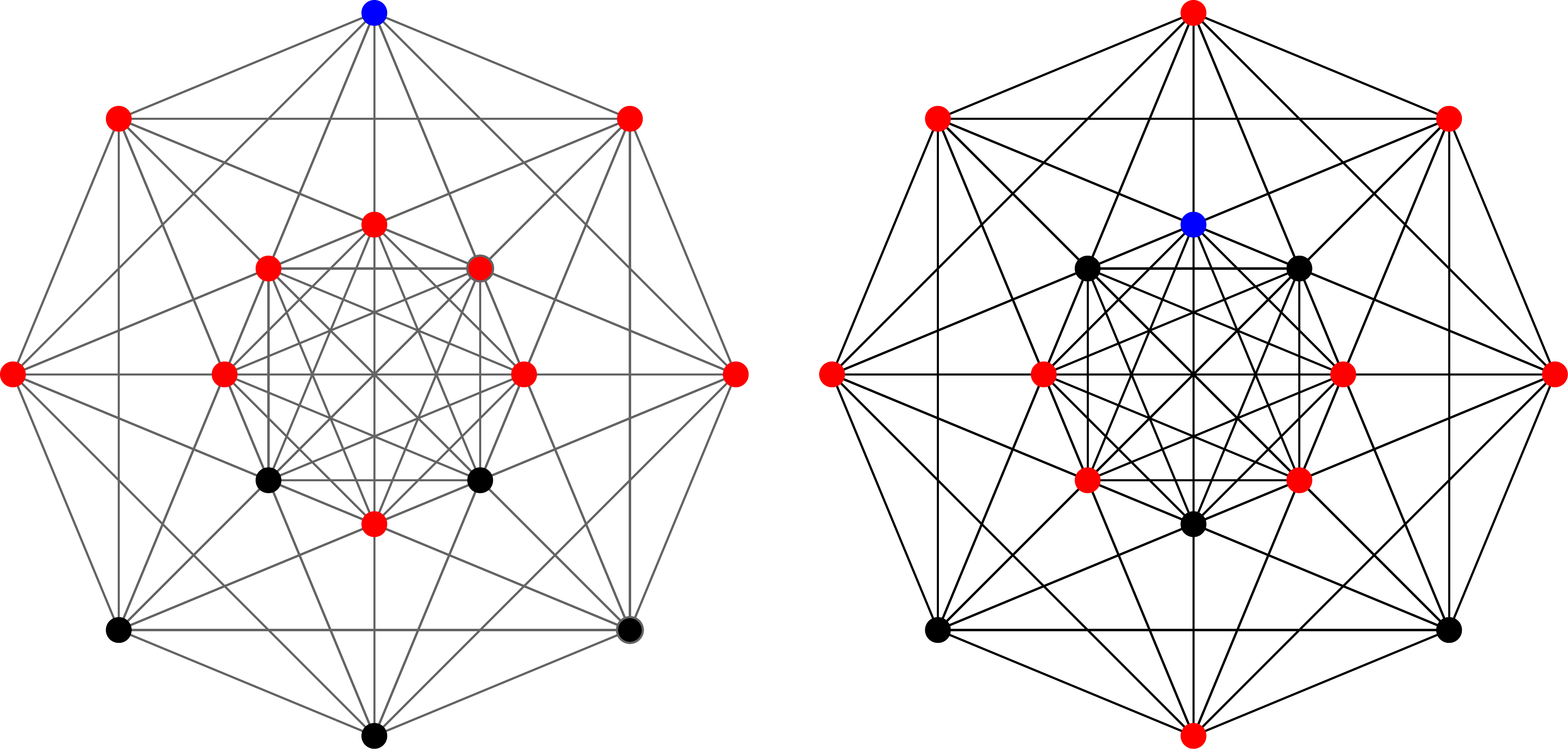}
 \end{center}
 \nota{The adjacency graph of the 16 facets of $P^5$. 
Each facet is identified as a string $\pm \pm \pm \pm \pm$ that indicates the vector $(\pm 1, \pm 1, \pm 1, \pm 1, \pm 1)$ orthogonal to the supporting hyperplane. These 16 vectors in $\matR^5$ are projected orthogonally onto the plane spanned by $u = (\sqrt 2, \sqrt 2, 2-\sqrt 2, 2-\sqrt 2, 0)$ and $v=(2-\sqrt 2, \sqrt 2-2, \sqrt 2, -\sqrt 2, 0)$. The picture shows the image of this projection, with an edge connecting every pair of points that represent adjacent facets. (This is the projection of the 1-skeleton of the  
Euclidean Gosset polytope $1_{21}$ combinatorially dual to $P^5$ \cite{G}.) 
Beware that some edges are superposed in the projection: to clarify this ambiguity, we have selected a blue point and painted in red the 10 points connected to it, in two cases (all the other cases are obtained by rotation).}
 \label{P5:fig}
\end{figure}

\subsection{The ideal vertices} 
The polytope $P^5$ has 10 ideal vertices $\pm e_1, \ldots, \pm e_5$. The link of each ideal vertex $v= \pm e_i$ is a 4-cube, whose facets are contained in the 8 facets of $P^5$ incident to $v$. These 8 facets are precisely those orthogonal to the vectors $(\pm 1, \pm 1, \pm 1, \pm 1, \pm 1)$ whose $i$-th coordinate is the same as that of $v$.

\begin{figure}
 \begin{center}
  \includegraphics[width = 12.5 cm]{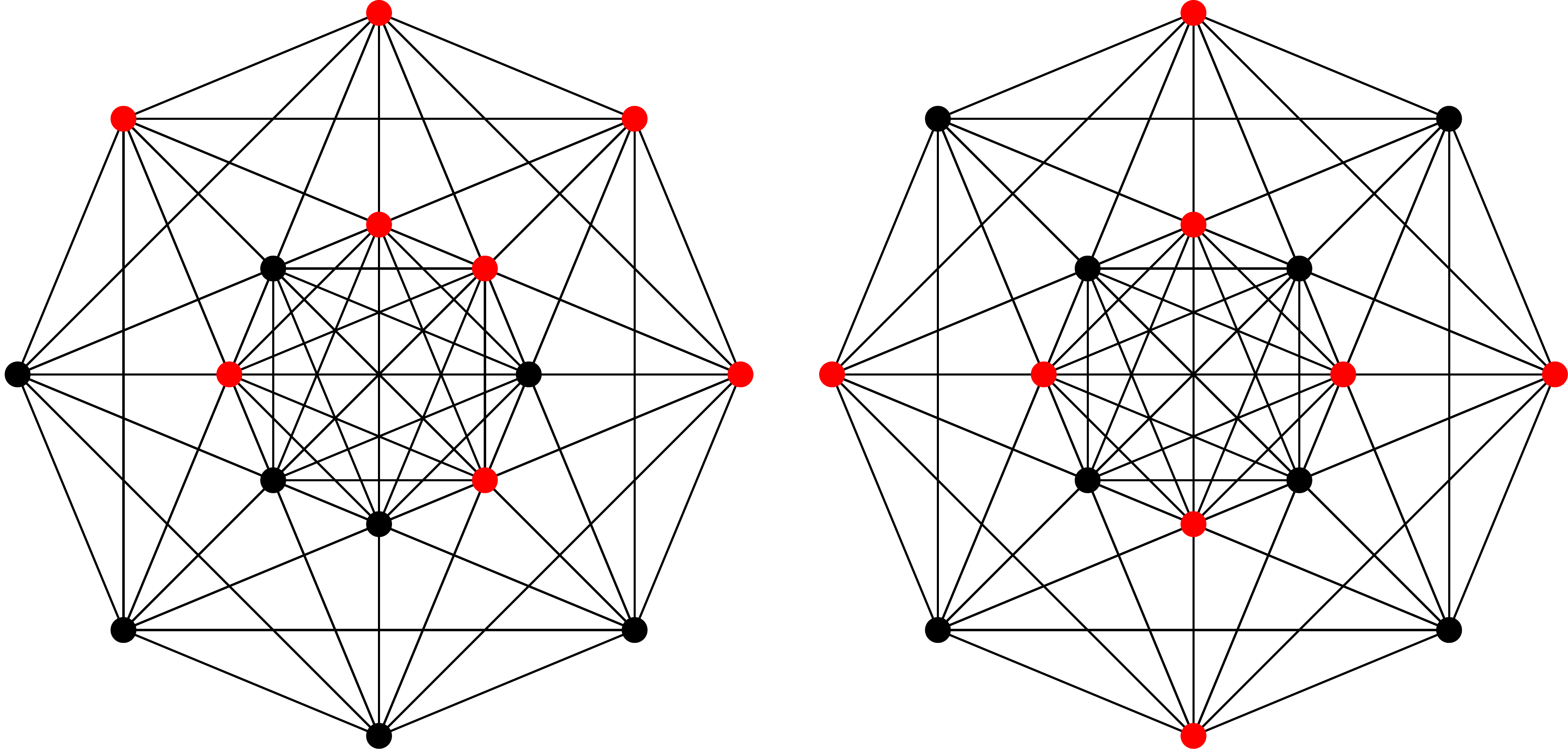}
 \end{center}
 \nota{Each ideal vertex $v$ of $P^5$ is determined by the 8 facets it is adjacent to, here drawn as red dots in the adjacency graph. If $i\leq 4$ the 8 facets are like in the left figure (we get 8 configurations of this type by rotation), while if $i=5$ they are like in the right figure (we get 2 configurations of this type by rotation).}
 \label{cusps:fig}
\end{figure}

The 8 facets incident to an ideal vertex $v= \pm e_i$ can be visualised in the adjacency graph of $P^5$ as in Figure \ref{cusps:fig}. If $i\leq 4$ the 8 facets are as in Figure \ref{cusps:fig}-(left), while if $i=5$ these are as in Figure \ref{cusps:fig}-(right). 


\subsection{The manifold $M^5$}
We constructed the hyperbolic manifold $M^5$ in \cite{IMM} by assigning a very symmetric colouring to $P^5$, that may be described as follows. 

We name each facet of $P^5$ as in Figure \ref{P5:fig} with the vector $(\pm 1, \ldots, \pm 1)$ orthogonal to its supporting hyperplane (there is an even number of minus signs). We assign 8 distinct colours to the facets of $P^5$, by giving the same colour to each pair
$$(\pm 1, \pm 1, \pm 1, \pm 1, \pm 1), \qquad (\mp 1, \mp 1, \mp 1, \mp 1, \pm 1).$$
More specifically, we use the set $\{1,\ldots, 8\}$ as a palette of colours and we assign 
them to the facets as indicated in Figure \ref{P5_colouring:fig}.
This is indeed a colouring for $P^5$, that is adjacent facets are always coloured differently (the two facets in each pair are not adjacent). 

\begin{figure}
 \begin{center}
  \includegraphics[width = 7 cm]{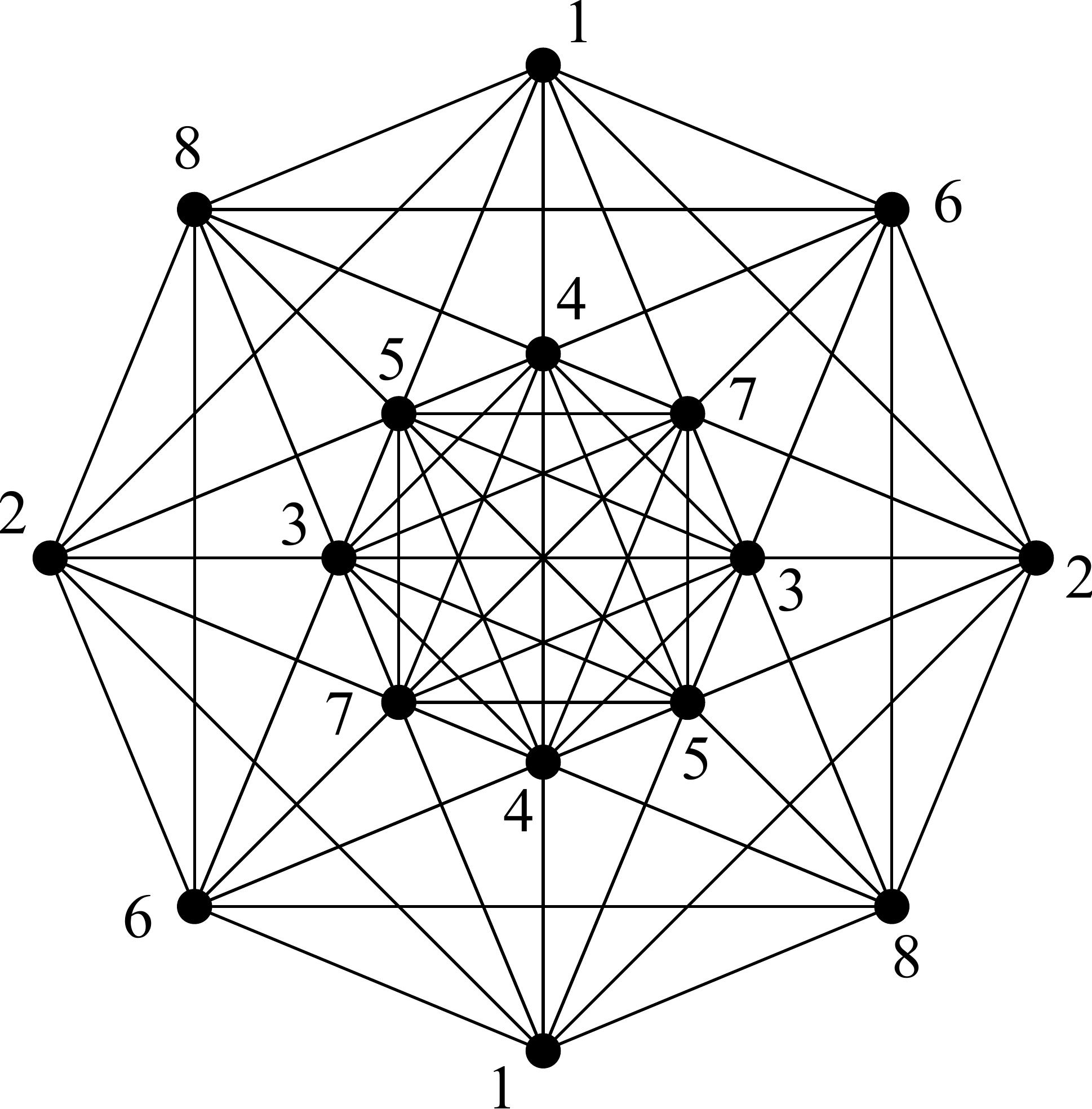}
 \end{center}
 \nota{The chosen colouring for $P^5$.}
 \label{P5_colouring:fig}
\end{figure}

We have then used this colouring to construct $M^5$ via a standard procedure  explained in \cite{IMM}. We pick $2^8$ copies of $P^5$ parametrised as $P_v^5$ with $v\in (\matZ_2)^8$. For every $v$ and every facet $F$ of $P^5$ having some colour $i \in \{1,\ldots, 8\}$, we identify the facet $F$ of $P_v^5$ with the same facet $F$ of $P_{v+e_i}^5$ via the identity map. See \cite{IMM} for more details.

After these identifications we get a cusped hyperbolic manifold $M^5$ tessellated into $2^8$ copies of $P^5$, with a natural orbifold covering $M^5 \to P^5$. It is shown in \cite{IMM} that $M^5$ has 40 cusps and Betti numbers
$$b_1(M^5) = 24, \quad b_2(M^5) = 120, \quad b_3(M^5) = 136, \quad b_4(M^5) = 39.$$

The Betti numbers are calculated using an algorithm from \cite{CP}.

\subsection{The 8 large and 32 small cusps}
To understand the cusps of $M^5$ it suffices to examine a 4-cube link $C$ at each ideal vertex $v$ of $P^5$, equipped with the colouring induced from that of the facets of $P^5$ incident to $v$ (each facet of $C$ inherits the colour of the facet of $P^5$ it is contained in). By the same procedure described above, the colouring of $C$ gives rise to a flat 4-manifold, which is always a 4-torus \cite[Proposition 7]{IMM}.
The counterimage of $C$ in $M^5$ is a disjoint union of $2^c$ such 4-tori, where $c$ is the number of colours in $\{1,\ldots,8\}$ that are \emph{not} present among those of the facets incident to $v$. So there are $2^c$ cusps in $M^5$ lying above $v$, where $c$ depends on $v$. See \cite{IMM} for more details.

\begin{figure}
 \begin{center}
  \includegraphics[width = 12.5 cm]{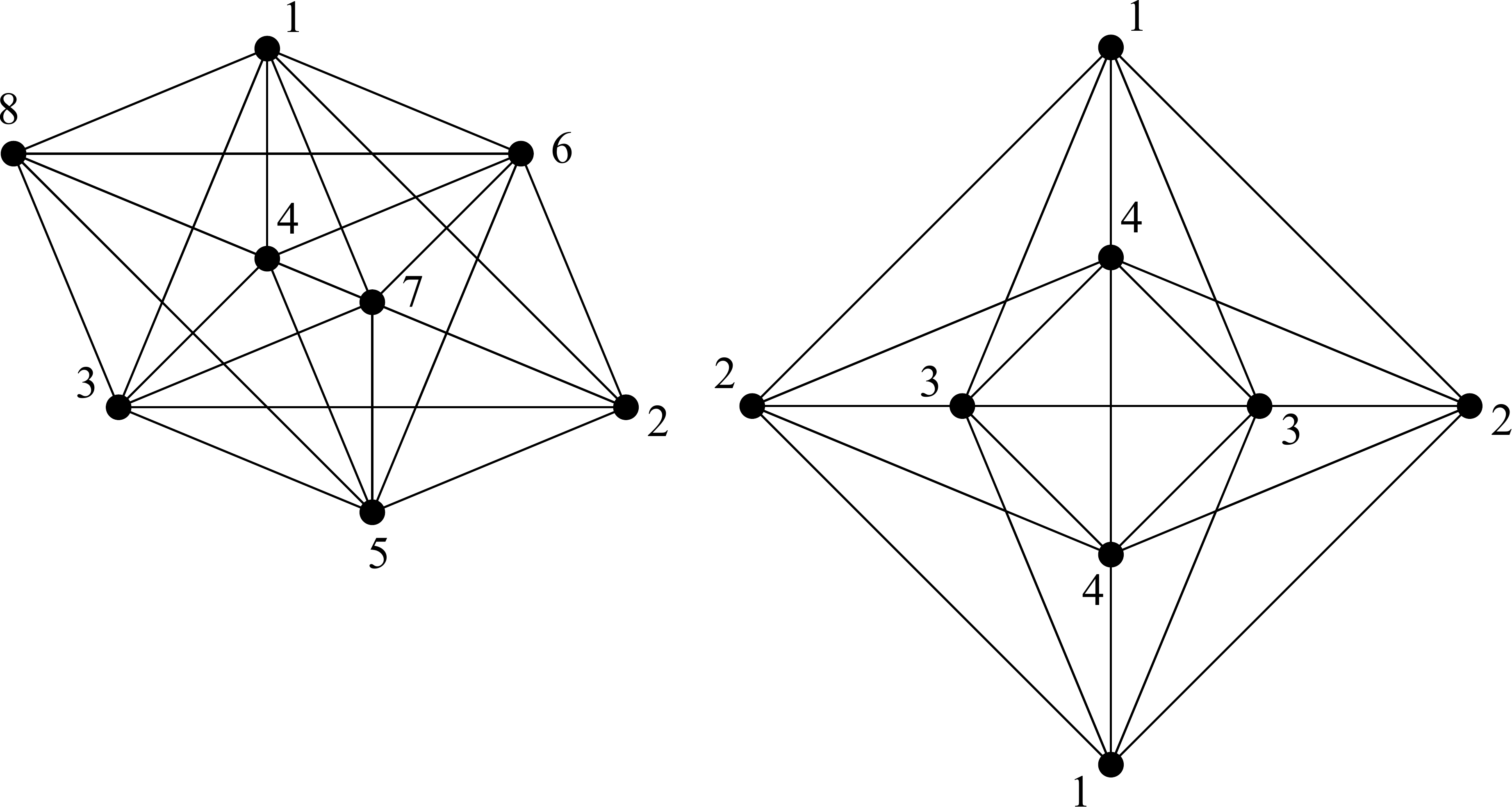}
 \end{center}
 \nota{The two types of ideal vertices of $P^5$ have a 8-coloured and a 4-coloured 4-cube link, respectively.}
 \label{P5_facets:fig}
\end{figure}

There are two types of ideal vertices in $P^5$, and we get $c=8$ for the first type and $c=4$ for the second: some cases are shown in Figure \ref{P5_facets:fig}. Here are the details.
Each of the 8 ideal vertices $v$ of type 
$$(\pm 1, 0,0,0,0), \quad \ldots \quad (0,0,0, \pm 1, 0)$$
is as in Figure \ref{cusps:fig}-(left), and
is adjacent to 8 facets of $P^5$ with all distinct colours: this is clear either by looking at Figure \ref{P5_colouring:fig}, or by recalling that these 8 facets are orthogonal to the vectors $(\pm 1, \pm 1, \pm 1, \pm 1, \pm 1)$ that share the same $i$-th coordinate $\pm 1$ if $v = \pm e_i$. All the 8 colours of the palette are employed and therefore there is a unique cusp in $M^5$ above $v$. The preimage of a 4-cube link $C$ for $v$ is a single 4-torus cusp section in $M^5$ tessellated into $2^8$ distinct 4-cubes. In fact \cite[Proposition 7]{IMM} implies that these form a $4\times 4 \times 4 \times 4$ hypercube with opposite facets identified by translations. We call the 8 cusps of $M^5$ obtained in this way (one above each of the 8 ideal vertices $v$) the \emph{large cusps}.

On the other hand, each of the remaining two ideal vertices $v$ of type 
$$(0,0,0,0, \pm 1)$$
is as in Figure \ref{cusps:fig}-(right), and is adjacent to 4 pairs of facets having the same colours. Therefore there are 4 colours in the palette $\{1,\ldots,8\}$ that are \emph{not} present among the facets adjacent to $v$. This implies that there are as many as $2^4$ cusps in $M^5$ lying above $v$. A 4-cube link $C$ for $v$ lifts to $2^8$ distinct 4-cubes, that form $2^4$ disjoint 4-tori cusp sections in $M^5$. By \cite[Proposition 7]{IMM} each 4-torus is tessellated into $2^4$ distinct 4-cubes like a  $2 \times 2 \times 2 \times 2$ hypercube with opposite faces identified. We call the $2 \times 16 = 32$ cusps of $M^5$ obtained in this way the \emph{small cusps}.

Summing up, the hyperbolic manifold $M^5$ has 8 large cusps and 32 small cusps, all \emph{toric} (that is with 4-tori sections) and all \emph{cubic} (that is obtained by identifying the opposite facets of a 4-cube). 
To be precise, the ``largeness'' of a cusp is not an intrinsic property of the cusp itself, but only of the sections that we have chosen. It is convenient to adopt these terms since these two types of cusps will sometimes play two different roles in the rest of the paper. 

\subsection{The dual cubulation} 
The cusped hyperbolic manifold $M^5$ is diffeomorphic to the interior of a compact smooth 5-manifold with boundary that we denote as $\bar M^5$. The boundary of $\bar M^5$ consists of 40 four-dimensional tori.

\begin{figure}
 \begin{center}
  \includegraphics[width = 12.5 cm]{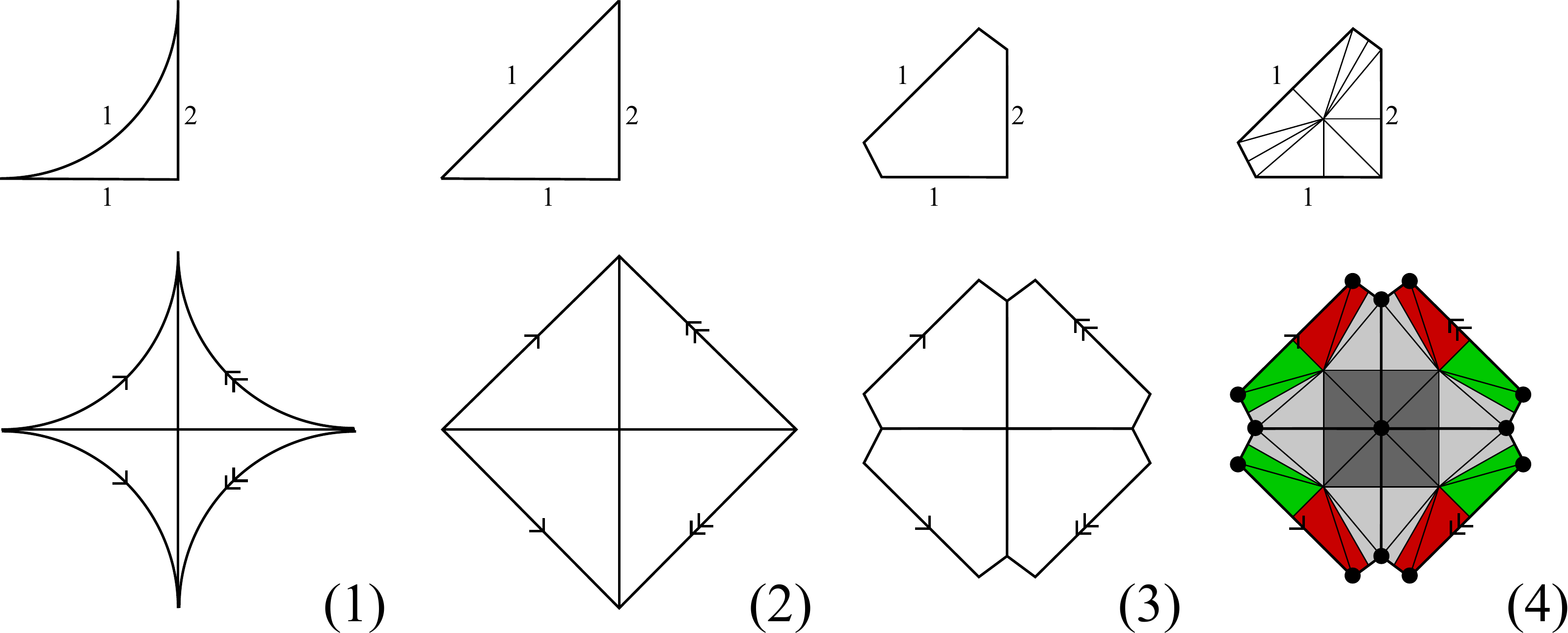}
 \end{center}
 \nota{The dual cubulation. In this example $P^2$ is a triangle with two ideal vertices and one real vertex with angle $\pi/2$, coloured with two colours 1 and 2. The colouring produces a thrice punctured sphere $M^2$ tessellated into four copies of $P^2$ (1). The same $P^2$ and $M^2$ are shown in the piecewise-linear setting (2), and then truncated to $\bar P^2$ and $\bar M^2$ (3). Now $\bar M^2$ is a pair-of-pants and the dual cubulation consists of 9 squares, one dual to the interior vertex of the tessellation and 8 dual to the boundary ones (4).}
 \label{dual_cubulation:fig}
\end{figure}

It is shown in \cite[Section 1.2]{BM} that the decomposition of $M^5$ into $2^8$ copies of $P^5$ induces a \emph{dual cubulation} for $\bar M^5$, considered as a piecewise-linear manifold. This is a general fact on manifolds obtained from coloured polytopes: a simple two-dimensional example is shown in Figure \ref{dual_cubulation:fig}. We recall here this general construction by focusing on $P^5$ for the sake of clarity.

Let $\bar P^5\subset \matR^5$ be the truncated $P^5$, that is the Euclidean polytope obtained by intersecting $P^5$ with the 10 half-spaces $\pm x_i \leq 1-\varepsilon$, for some fixed small $\varepsilon > 0$. The polytope $\bar P^5$ is obtained from $P^5$ by removing some small stars of its ideal vertices. 

The Euclidean polytope $\bar P^5$ has $16+10$ facets: these are the truncated 16 facets of $P^5$ plus 10 additional small hypercubes produced by the truncation of the ideal vertices. It has $16+10 \times 16$ vertices, that are the 16 original finite vertices of $P^5$ plus the vertices of the 10 new hypercubes.

We equip the 16 truncated facets of $\bar P^5$ with the same colouring of the corresponding facets in $P^5$, and leave the small 10 hypercubes uncoloured.
We apply the same colouring technique described above to $\bar P^5$, that is we take $2^8$ copies $\bar P_v^5, v \in (\matZ_2)^8$ of $\bar P^5$ and we glue a facet $F$ of $\bar P_v^5$ with the same facet $F$ of $\bar P_{v+e_i}^5$ if $F$ has colour $i$. The small uncoloured hypercubes are left unglued, as in Figure \ref{dual_cubulation:fig}-(3).

As a result we get a compact piecewise-linear manifold $\bar M^5$ with boundary, whose interior is easily seen to be homeomorphic to $M^5$. 
We now construct a cubulation for $\bar M^5$, that is a decomposition into 5-cubes.
The manifold $\bar M^5$ is tessellated into $2^8$ copies of $\bar P^5$. We fix a barycentric subdivision for $\bar P^5$, which lifts to a barycentric subdivision of the tessellation of $\bar M^5$. 
We consider the vertices $v$ of the \emph{original} (not yet subdivided) tessellation of $\bar M^5$. Their stars $\St(v)$ in the barycentric decomposition form a dual decomposition of $\bar M^5$ into polytopes as in Figure \ref{dual_cubulation:fig}-(4). Since $P^5$ is right-angled, the star $\St(v)$ can be of two types, depending on whether $v$ lies in the interior or in the boundary of $\bar M^5$. Let $\pi(v)$ be the image of $v$ in $\bar P^5$. Then:

\begin{itemize}
\item if $\pi(v)$ is one of the 10 real vertices of $P^5$, then $v$ lies in the interior of $\bar M$ and $\St(v)$ is the barycentric subdivision of a 5-cube, with $v$ at its center;
\item if $\pi(v)$ is one of the additional $10\times 16$ vertices of $\bar P^5$ obtained from the truncation, then $v \in \partial \bar M^5$ and $\St(v)$ is one half of a barycentric subdivision of a 5-cube, with $v$ at its center.
\end{itemize}

See Figure \ref{dual_cubulation:fig}. In both cases we get a 5-cube, and the 5-cubes so obtained intersect along common facets, dual to the edges of the decomposition of $\bar M^5$ into copies of $\bar P^5$. 
We have thus obtained a decomposition of the piecewise-linear manifold $\bar M^5$ into 5-cubes, that is a \emph{cubulation}. 


\subsection{States and moves}
We explain a general strategy to construct interesting piecewise-linear circle-valued maps using the cubulation of $\bar M^5$ just introduced. The strategy is taken from \cite{BB} and \cite{JNW}, with some modifications, and has already been used in \cite{BM, IMM}. It consists of a combinatorial game based on some \emph{states} and \emph{moves}, first introduced by Jankiewicz, Norin, and Wise in \cite{JNW}.

Let a \emph{state} of $P^5$ be the partition of its facets into two subsets, that we denote as I (in) and O (out). Every facet inherits a \emph{status} I or O. Analogously, we define a state of $\bar P^5$ as the partition of its non-cubic facets into I and O, so a state on $P^5$ obviously induces one on $\bar P^5$ and viceversa, and we will pass from one to the other without mention.

Let a \emph{set of moves} $\calS$ be an abritrary partition of the set $\{1,\ldots,8\}$ of colours used for $P^5$. Each set $S\in \calS$ of the partition is a \emph{move}. (This definition is slightly different from that given in \cite{JNW}. The main difference is that here every partition is allowed: we do not require all the facets coloured with some colour in a fixed move $S\in \calS$ to be disjoint.) For instance, the set of moves that we use here is the following:
$$
\calS = \big\{\{1,5\}, \{2,6\}, \{3,7\}, \{4,8\}\big\}.
$$

Moves act on states as follows. Given a move $S\in \calS$ and a state $s$, we build a new state by switching the I/O status of all the facets whose colour belongs to $S$.

We now fix a state $s$, that we call the \emph{initial state}, and a set of moves $\calS$. We now show that the pair $(s,\calS)$ induces an orientation on all the edges of the dual cubulation of $\bar M^5$ described above.

Recall that $M^5$ decomposes into the polytopes $P^5_v$ as $v\in \matZ_2^8$ varies. We assign a state $s_v$ to each polytope $P^5_v$ as follows. We start with the initial state $s$. Then, for every move $S\in \calS$, we calculate the sum 
$$\sigma = \sum_{i\in S} v_i \in \matZ_2$$
where $v_i$ is the $i$-th component of $v\in \matZ_2^8$. If $\sigma=1$ we let $S$ act on $s$, that is we swap all the stati I/O of all the facets of $P^5_v$ having a colour that belongs to $S$. If $\sigma=0$, we do nothing. Let $s_v$ be the final state that we get after all these modifications. Note that in particular we have $s_0 = s$, so the initial state is assigned to $P^5_0$.

If we travel from $P^5_v$ to $P^5_{v+e_i}$ by crossing a facet $F$ coloured with $i$, we notice that the status of $F$ is necessarily changed, because the parity of the corresponding $\sigma$ changes. More precisely: 

\begin{center}
\emph{When we cross $F$, the stati of all the facets whose colour belongs \\
to the same move $S$ containing the colour of $F$ are changed. \\
The stati of the other facets are left unaltered.}
\end{center}

If we interpret I and O as ``In'' and ``Out'', we can orient the facet $F$ of $P_v^5$ transversely, with a vector that goes from the adjacent polytope where the status of $F$ is O to the one where its status is I. We orient the edge of the cubulation dual to $F$ accordingly. See an example in Figure \ref{orientations:fig}-(centre).

We have oriented in this way all the edges of the cubulation that are dual to some facet of some $P_v^5$, that is all the edges that are fully contained in the interior of $\bar M^5$. There are two more types of edges to orient:
\begin{itemize}
\item those that connect an interior vertex of the cubulation with a boundary vertex are oriented towards the boundary;
\item those that are fully contained in $\partial \bar M^5$ are adjacent to a unique square that is not contained in $\partial \bar M^5$, and we orient them like the opposite edge in that square.
\end{itemize}

We have oriented all the edges of the cubulation. The orientation of course depends on the initial state $s$ and on the set of moves $\calS$. An example is shown in Figure \ref{orientations:fig}-(right).

\begin{figure}
 \begin{center}
  \includegraphics[width = 11 cm]{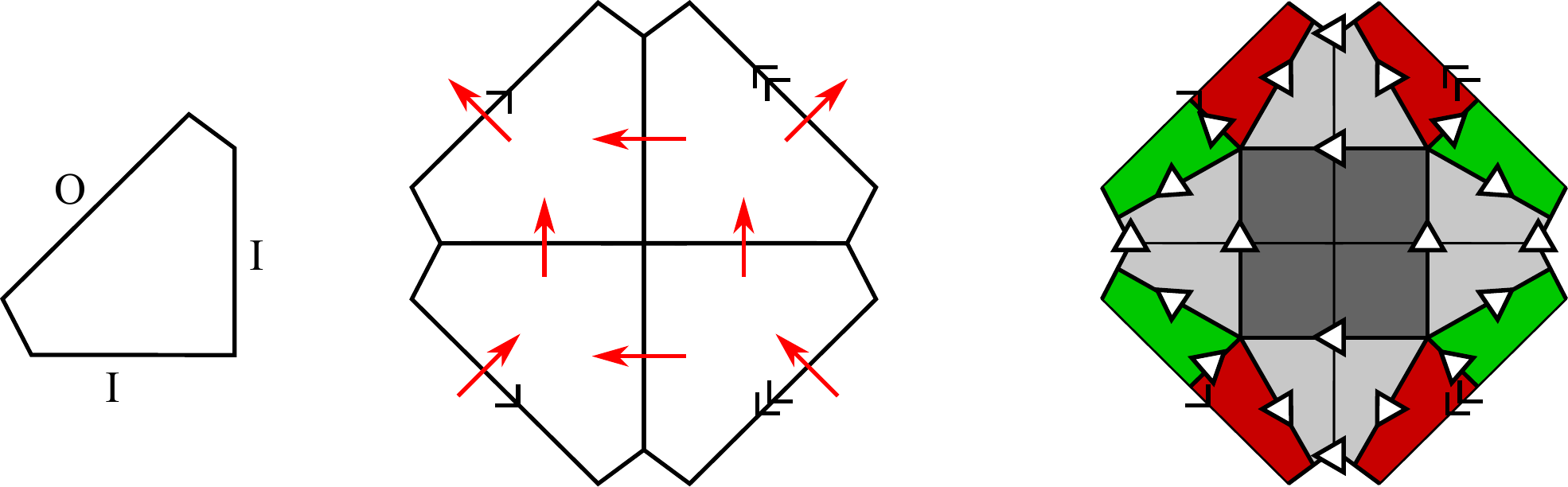}
 \end{center}
 \nota{We consider the coloured polygon of Figure \ref{dual_cubulation:fig}. We fix an initial state $s$ (left), and we take $\calS=\{\{ 1\}, \{2\}\}$ as a set of moves. The facets of the decomposition inherit the transverse orientation shown here (centre). The edges of the dual cubulation also inherit orientations (right). In this case all the squares are coherently oriented.}
 \label{orientations:fig}
\end{figure}

\subsection{Coherent orientations}
We say that an orientation of all the edges of a cubulation is \emph{coherent} if, on every square, opposite edges have the same orientations. For instance, the orientation shown in Figure \ref{orientations:fig}-(right) is coherent.

A coherent orientation is a very useful tool because it defines a circle-valued \emph{diagonal map} as follows: we identify every $k$-cube of the cubulation with $[0,1]^k$, in a way that every edge is oriented like the natural orientation of $[0,1]$. Then we define on the $k$-cube the diagonal map
$$f(x_1, \ldots, x_k) = x_1 + \cdots + x_k.$$ 
All these maps glue to a piecewise-linear map from the cubulation to $\matR/\matZ$. The resulting map is a (circle-valued) \emph{Morse function} in the sense of \cite{BB}. All the vertices of the cubulation are sent to zero, which should be interpreted as the unique singular value of the Morse function.

Summing up, a coherent orientation on the edges defines a piecewise-linear circle-valued Morse function. We thus aim at constructing coherent orientations.

\begin{defn}
A set of moves $\calS$ is \emph{sparse} if for every move $S\in \calS$ the facets of $P^5$ coloured with some $i\in S$ are pairwise disjoint.
\end{defn}

We call a 3-dimensional face of $P^5$ a \emph{ridge}. Let $s$ be an initial state and $\calS$ be a set of moves.

\begin{prop}  \label{sparse:prop}
The orientation on the edges of the cubulation induced by $s$ and $\calS$ is coherent if and only if $\calS$ is sparse.
\end{prop}
\begin{proof}
By construction, all the squares that intersect $\partial \bar M^5$ in a single edge are oriented coherently. All the squares $Q$ that are contained in $\partial \bar M^5$ are oriented like the opposite square $Q'$ in the unique cube that intersects $\partial \bar M^5$ in $Q$, and $Q'$ is contained in the interior of $\bar M^5$. 

\begin{figure}
 \begin{center}
  \includegraphics[width = 9 cm]{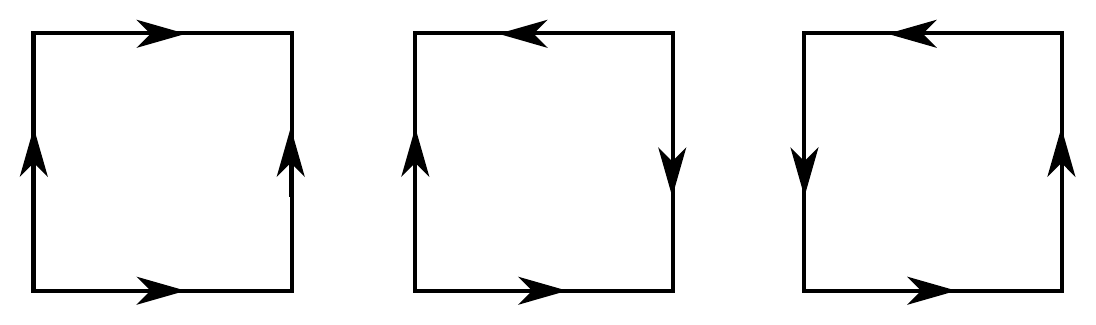}
 \end{center}
 \nota{Three types of squares with oriented edges. Only the first is oriented coherently.}
 \label{squares:fig}
\end{figure}

Therefore it suffices to consider the squares that are contained in the interior of $\bar M^5$. Each such is dual to some codimension-2 stratum of the polyhedral decomposition of $M^5$, which projects to a ridge of $P^5$, intersection of two facets $F_1, F_2$. If the colours of $F_1$ and $F_2$ belong to distinct moves in $\calS$, the square is oriented coherently, because when we cross $F_1$ the status of $F_2$ is left unchanged. If not, the square has one of the two forbidded orientations of Figure \ref{squares:fig}-(center and right).
\end{proof}

\subsection{Ascending and descending links}
We keep using the terminology of \cite{BB}. A cube complex is a particular kind of \emph{affine cell complex}, see \cite[Definition 2.1]{BB}, where every $k$-cube is given the structure of a Euclidean unit cube.
Let $f\colon X \to \matR/\matZ$ be a circle-valued \emph{Morse function} in the sense of \cite[Definition 2.2]{BB} defined on some affine cell complex $X$. Let $\lk(v)$ be the link of a vertex $v$ in $X$. Recall from \cite[Definition 2.4]{BB} that the \emph{ascending} (\emph{descending}) \emph{link} $\lk_\uparrow(v)$ ($\lk_\downarrow(v)$) is the union of the links of all the cells incident to $v$ where $f$ has a minimum (maximum) at $v$.

An affine cell complex has a natural PL structure \cite[Page 448]{BB}. When we decompose a smooth manifold $W$ as an affine cell complex we always suppose implicitly that the PL structures coming from the smooth and the cell complex structures are compatible. This will be the case with the examples considered here in this paper.

\begin{teo}
Let $W$ be a compact smooth manifold of dimension $n\leq 5$, possibly with boundary. Let $W$ be decomposed as an affine cell complex and $f\colon W \to \matR/\matZ$ be a circle-valued Morse function in the sense of \cite{BB}. If the ascending and descending links at all the vertices are collapsible the map $f$ is homotopic to a smooth fibration.
\end{teo}
\begin{proof}
The same proof as \cite[Theorem 15]{BM} works here: the counterimage of every closed interval with endpoints both distinct from $0\in \matR/\matZ$ is piecewise-linearly a product. We can thus homotope $f$ so that it is a piecewise-linear fibration, that is the counterimage of every interval is a product. The dimension bound $n\leq 5$ ensures that $f$ can be smoothened to a fibration \cite{HM, Mun}.
\end{proof}

Our goal is now clear: we would like to find an initial state $s$ and a set of moves $\calS$ for $P^5$ such that
\begin{enumerate}
\item $\calS$ is sparse, and
\item all the ascending and descending links are collapsible polyhedra.
\end{enumerate}

Unfortunately, it turns out that the only sparse $\calS$ is the partition into singletons 
$$\calS_0 = \{\{1\}, \{2\}, \{3\}, \ldots, \{8\}\}.$$
We initially picked $\calS_0$ as a set of moves and searched with a computer for an initial state $s$ that would fulfill the requirement (2), but we could not find any. After some time, we realised that no such state $s$ can exist! One can prove quite easily that with such a $\calS_0$ every state $s$ produces a circle-valued map $f$ that is homotopically trivial on the large cusps (this phenomenon occurs for instance for the state chosen in \cite{IMM}), and hence it will never be homotopic to a fibration.

Despite these negative news, we could find one interesting pair $(s, \calS)$ which fulfills  (2), although $\calS$ is not sparse. As anticipated above, the set of moves is
$$
\calS = \big\{\{1,5\}, \{2,6\}, \{3,7\}, \{4,8\}\big\}.
$$
We now describe the initial state $s$. We will later deal with the fact that $\calS$ is not sparse: this problem will be resolved by subdividing some cubes of the cube complex into prisms.

\subsection{The balanced states}
Every move $S\in\calS$ is of type $S=\{t,t+4\}$ for some fixed $t$. There are four facets of $P^5$ whose colour is in $S$, and we call them a \emph{quartet}. The 16 facets of $P^5$ are thus subdivided by $\calS$ into four quartets.

\begin{figure}
 \begin{center}
  \includegraphics[width = 7 cm]{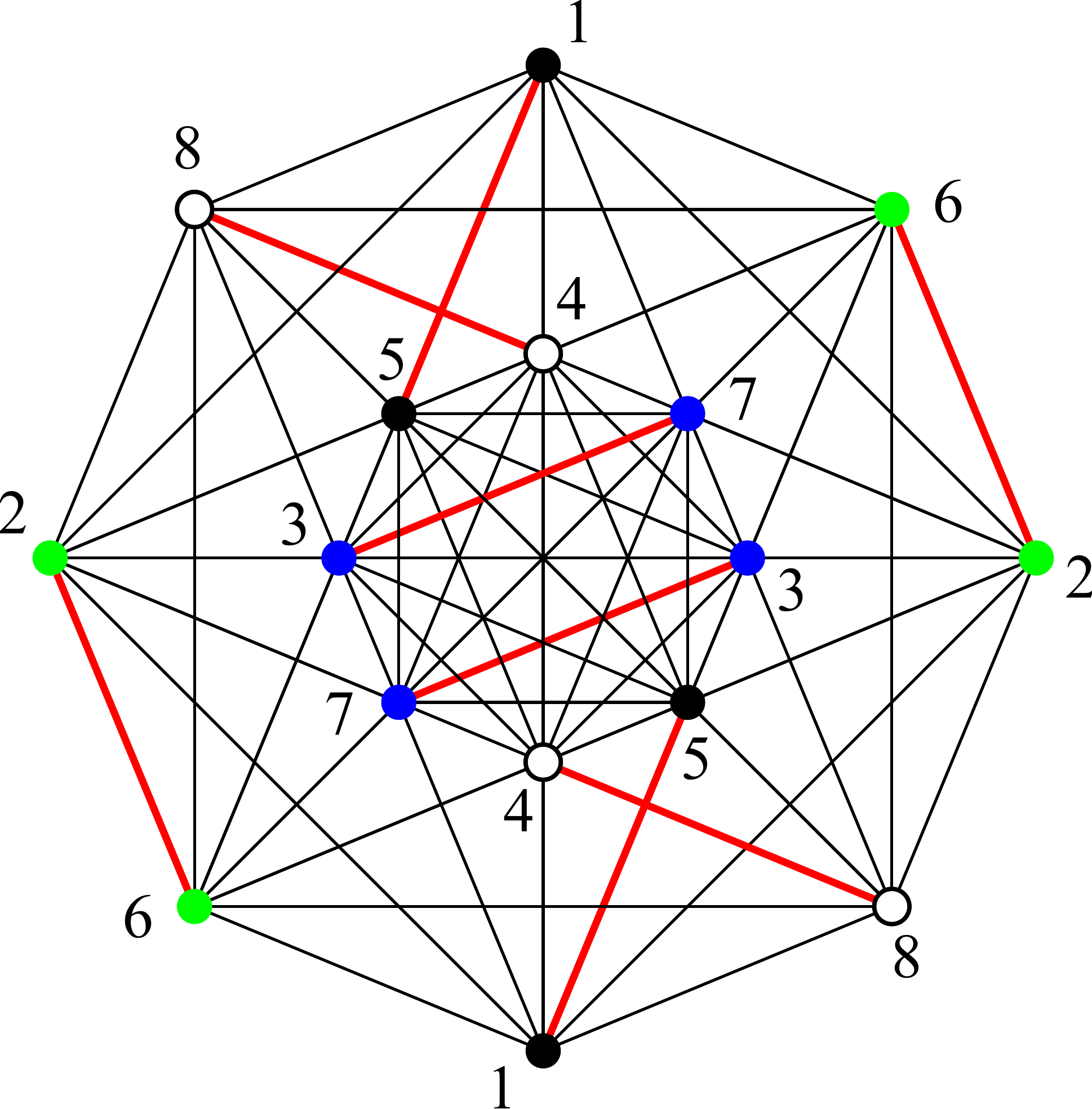}
 \end{center}
 \nota{The four quartets of vertices, depicted with four different colours. Each quartet is subdivided into two pairs of adjacent facets. The adjacencies are indicated here as red edges.}
 \label{quartets:fig}
\end{figure}

Figure \ref{quartets:fig} shows that each quartet is subdivided into two pairs of adjacent facets, each pair consisting of facets with distinct colours $t$ and $t+4$ whose coordinates $\pm \pm \pm \pm \pm$ differ only at the positions $t$ and $5$, see Figure \ref{P5:fig}. 
We call one such pair an \emph{adjacent pair}. So each quartet is subdivided into two adjacent pairs. We say that a state is \emph{balanced} if, on each quartet, it assigns the status I to one adjacent pair and O to the other. By applying the moves in $\calS$ we act transitively on the set of all the $2^4$ balanced states.

\begin{figure}
 \begin{center}
  \includegraphics[width = 7 cm]{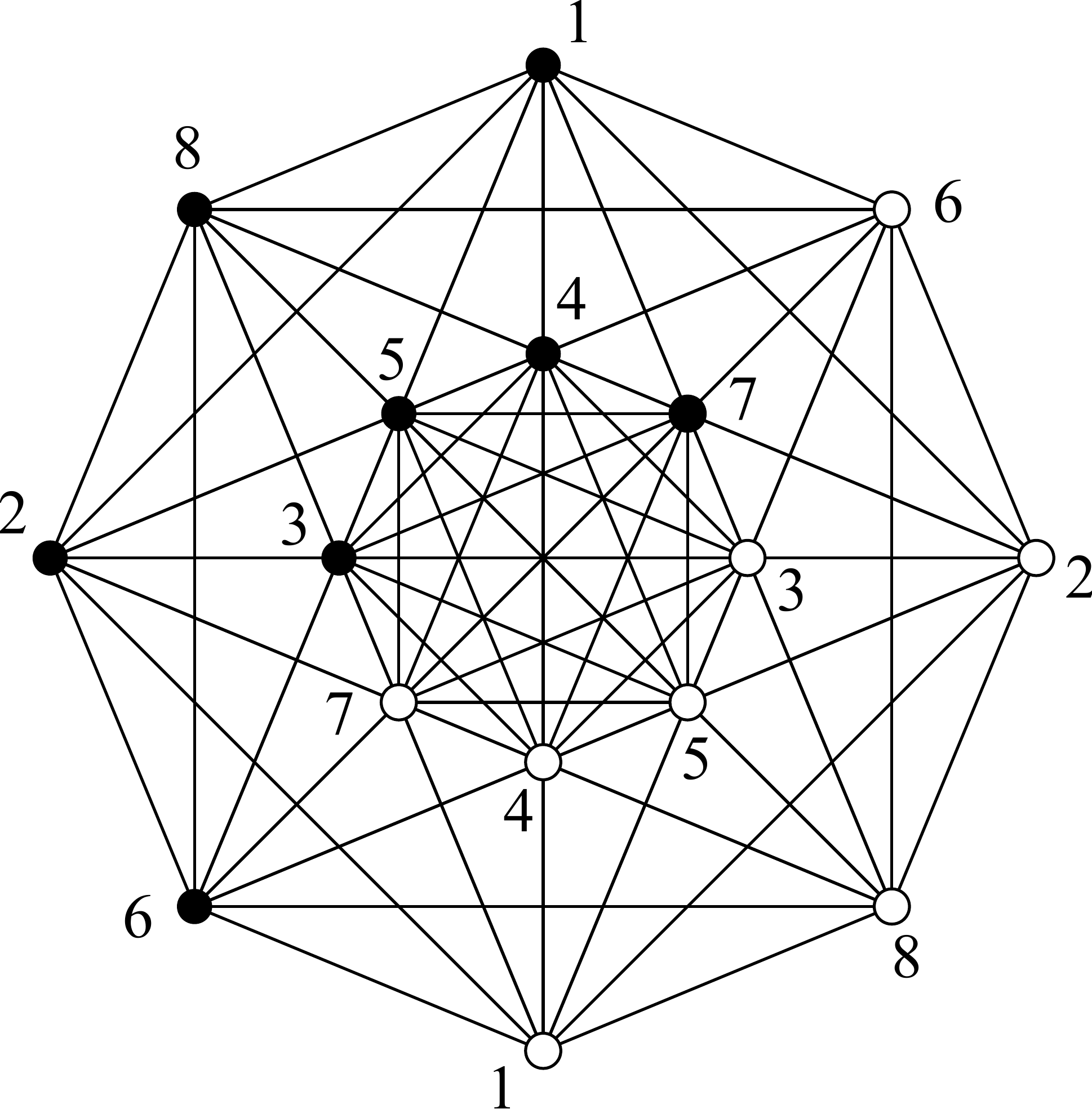}
 \end{center}
 \nota{The chosen initial balanced state $s$ for $P^5$. The black (white) dots indicate the status I (O).}
 \label{s:fig}
\end{figure}

We let the initial state $s$ be the balanced state shown in Figure \ref{s:fig} (any other balanced state would do the job, we fix one just for convenience). 
We use the set of moves $\calS$ to assign a state $s_v$ on $P^5_v$ for each $v\in \matZ_2^8$ as explained before. The states $s_v$ assigned to $P^5_v$ as $v\in \matZ_2^8$ varies are precisely all the $2^4$ balanced states on $P^5_v$, each repeated $2^4$ times. 

The stati in $P^5_v$ of the four facets coloured with either $t$ or $t+4$ are those of the initial state $s$ if $v_t+v_{t+4}$ is even, and reversed otherwise. Therefore the polytopes $P^5_v$ and $P^5_w$ will inherit the same state $s_v=s_w$ precisely when 
$$v_t+v_{t+4} \equiv w_t+w_{t+4} \mod 2 \quad \forall t.$$

\subsection{Sixteen useful symmetries of $P^5$} 
We now introduce 16 symmetries of $P^5$ that will be particularly useful. Recall that $\Isom(P^5)$ has order $1920$ and consists of all the maps
$$(x_1,x_2,x_3,x_4,x_5) \longmapsto (\pm x_{\sigma(1)}, \pm x_{\sigma(2)}, \pm x_{\sigma(3)}, \pm x_{\sigma(4)}, \pm x_{\sigma(5)})$$
where $\sigma \in S_5$ is any permutation and there is an even number of minus signs.
We identify $\matR^5$ with $\matH \times \matR$ where $\matH$ is the quaternions space, and consider the group 
$$Q_8 = \{\pm 1, \pm i, \pm j, \pm k\}.$$ 
If we let $Q_8$ act on the first factor of $\matH \times \matR$ by left multiplication, we obtain the following isometries of $P^5$:
\begin{align*}
L_{\pm 1}\colon (x_1,x_2,x_3,x_4, x_5) \longmapsto (\pm x_1, \pm x_2, \pm x_3, \pm x_4, x_5), \\
L_{\pm i} \colon (x_1,x_2,x_3,x_4, x_5) \longmapsto (\mp x_2, \pm x_1, \mp x_4, \pm x_3, x_5), \\
L_{\pm j} \colon (x_1,x_2,x_3,x_4, x_5) \longmapsto (\mp x_3, \pm x_4, \pm x_1, \mp x_2, x_5), \\
L_{\pm k} \colon (x_1,x_2,x_3,x_4, x_5) \longmapsto (\mp x_4, \mp x_3, \pm x_2, \pm x_1, x_5).
\end{align*}
Here $L_q$ denotes left-multiplication by $q\in Q_8$ on the first factor. We can thus see $Q_8$ as a subgroup of $\Isom(P^5)$. We are also interested in the isometric involution
$$\iota \colon (x_1,x_2,x_3,x_4,x_5) \longmapsto (x_1,-x_2,-x_4,-x_3,-x_5).$$

By direct inspection we see that $\iota$ normalises $Q_8$ and that the subgroup $R_{16}$ generated by $Q_8$ and $\iota$ has order 16. One may also prove that $R_{16}$ is isomorphic to the semidihedral group of order 16, but we will not need that. The group $R_{16}$ is particularly useful because of the following.

\begin{prop} \label{16:prop}
The group $R_{16}$ preserves (as partitions) both the colouring and $\calS$. It acts freely and transitively on the 16 facets and on the 16 balanced states of $P^5$. 
\end{prop}
\begin{proof}
We can verify easily that $Q_8$ preserves both the colouring and the partition, and the same holds for $\iota$. The group $Q_8$ acts freely and transitively on the 8 facets represented by a vector $(\pm 1, \pm 1, \pm 1, \pm 1, \pm 1)$ with the same fifth component. The element $\iota$ switches the sign of the fifth component, so $R_{16}$ acts transitively on all the 16 facets.

To deduce that $R_{16}$ acts transitively on the 16 balanced states, we note that there is a natural 1-1 correspondence between balanced states and facets, defined as follows: a balanced state is uniquely determined by the unique facet with status I that is incident to all the other facets with status I. For instance, in Figure \ref{s:fig} such a facet is the one with colour 8 and status I.
\end{proof}

Every isometry $\varphi \in R_{16}$ permutes the set of colourings $\{1,\ldots, 8\}$ and sends the $i$-coloured pair of facets to the $\varphi(i)$-coloured pair, either by preserving or by reversing the status of both. We encode this information by adding a bar over the $i$-th number on the permutation string if the stati are reversed. The identity of course acts as the permutation
$(1\ 2\ 3\ 4\ 5\ 6\ 7\ 8)$ and by direct inspection we see that $L_{-1}$ acts like
$$(\bar 1\ \bar 2\ \bar 3\ \bar 4\ \bar 5\ \bar 6\ \bar 7\ \bar 8).$$
We also check that $L_i, L_j, L_k$ act respectively as
$$(2\ 1\ 4\ 3\ 6\ 5\ 8\ 7), \qquad (\bar 3\ 4\ \bar 1\ 2\ \bar 7\ 8\ \bar 5\ 6), \qquad (4\ \bar 3\ \bar 2\ 1\ 8\ \bar 7\ \bar 6\ 5)
$$
while the involution $\iota$ acts as
$$(\bar 5\ 6\ 8\ 7\ \bar 1\ 2\ 4\ 3).$$
The position of the bars depends on the choice of our initial balanced state $s$, but the following fact is true for every possible choice: the isometries in $Q_8$ are precisely those with a number of bars that is divisible by four.

\subsection{The ascending and descending links are collapsible} 
Let a \emph{$k$-octahedron} be the regular polytope dual to the $k$-cube (sometimes also called $k$-orthoplex). The \emph{Gosset polytope} $1_{21}$ is a semiregular Euclidean polytope that is combinatorially dual to $P^5$, discovered by Gosset \cite{G} in 1900. It has 26 facets: 16 facets are 4-simplexes dual to the real vertices of $P^5$, while 10 facets are 4-octahedra dual to the ideal vertices of $P^5$. 

The 1-skeleton of $1_{21}$ is the adjacency graph of $P^5$. If we consider the boundary of $1_{21}$ and remove the ten 4-octahedra from it, we are left with a 4-dimensional simplicial complex $K$ which is isomorphic to the flag simplicial complex generated by the adjacency graph of $P^5$.

Given a state on $P^5$, we may define its \emph{ascending} (\emph{descending}) \emph{link} as the subcomplex of $K$ spanned by the vertices dual to the facets with status O (I). 

Here is a crucial property of the balanced states:

\begin{prop} \label{balanced:collapsible:prop}
On a balanced state, the ascending and descending links are both collapsible.
\end{prop}
\begin{proof}
It suffices to prove this for the balanced state $s$ of Figure \ref{s:fig}, since any other balanced state is related to it by an isometry of $P^5$ by Proposition \ref{16:prop}.
The collapsibility of this state was already noted in \cite{IMM}. The ascending and descending links both collapse to the vertex coloured by 8, because it is connected to every other vertex (so every maximal simplex contains it). 
\end{proof}

We now show that the same holds at the ideal vertices of $P^5$. The link of an ideal vertex $v$ of $P^5$ is a 4-cube, and it inherits a state from the initial state $s$ of $P^5$: every facet of the 4-cube gets the status I/O of the corresponding facet of $P^5$. The ascending and descending links of the 4-cube are defined analogously, as subcomplexes of the dual 4-octahedron.

\begin{prop} \label{vertices:coll:prop}
On a balanced state, the ascending and descending links at every ideal vertex of $P^5$ are collapsible.
\end{prop}
\begin{proof}
By looking at Figures \ref{cusps:fig} and \ref{P5_colouring:fig} we find that every ideal vertex $v$ of $P^5$ is adjacent to two mutually non-adjacent facets of $P^5$ that belong to the same quartet. These two facets are always assigned opposite stati. Therefore in the 4-cube link there are two opposite facets with opposite stati. Hence in the dual 4-octahedron there are two opposite vertices with opposite stati. Both the ascending and descending links collapse onto these two vertices.
\end{proof}

\subsection{The subdivided cell complex} \label{subdivided:subsection}
The initial state $s$ that we have chosen is very nice, but $\calS$ is not sparse. The induced orientation on the edges of the cubulation of $\bar M^5$ will contain some bad squares like in Figure \ref{squares:fig}-(center or right) and there is no way to define on such a square an affine function that is coherent with the orientation of the edges. We solve this problem by subdividing the cube complex into smaller pieces: we cut all the cubes that contain these bad squares into four prisms. 

We first determine the bad squares. Recall that a ridge in $P^5$ is a 3-dimensional face. We call a ridge \emph{bad} if the colours of the adjacent facets belong to the same quartet. The 8 bad ridges of $P^5$ are dual to the 8 red edges shown in Figure \ref{quartets:fig}. The bad squares contained in the interior of $\bar M^5$ are precisely those that are dual to some 3-stratum that projects to a bad ridge of $P^5$, as explained in the proof of Proposition \ref{sparse:prop}. Since the stati of the two facets incident to a bad ridge always coincide, the bad square is as in Figure \ref{squares:fig}-(center). The other bad squares in $\bar M^5$ are those contained in the boundary that are parallel to one interior bad square through a 3-cube.

\begin{figure}
 \begin{center}
  \includegraphics[width = 2.5 cm]{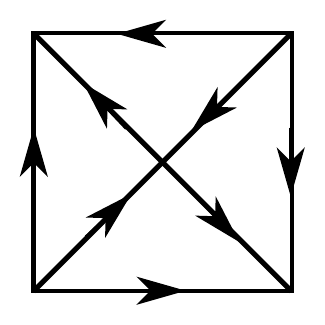}
 \end{center}
 \nota{A subdivision of a bad square into four triangles. On each triangle we have a natural affine height function inducing the orientation of its edges.}
 \label{subdivision:fig}
\end{figure}

The only bad squares are thus as in Figure \ref{squares:fig}-(center). This is good news, because these bad squares can be subdivided into four triangles as in Figure \ref{subdivision:fig}, and on each triangle $T$ there is a natural affine height function $T \to \matR/\matZ$ that induces the orientation of its edges and sends the three vertices to $0, 1/2,$ and $1=0$. The four affine functions match along the edges and the central vertex of the bad square is sent to $1/2$ (as a kind of piecewise-linear saddle point).

Given two cubes $C, C'$ with oriented edges, there is an obvious way to orient all the edges of $C\times C'$, since these are either of the form $e \times \{p'\}$ for some edge $e\subset C$ or $\{p\} \times e'$ for some edge $e'\subset C'$. A product of cubes with coherent orientations is a cube with coherent orientation.

A $k$-cube of the cubulation of $\bar M^5$ is called \emph{good} if it is coherently oriented and \emph{bad} otherwise. We already know that $\bar M^5$ contains only one type of bad squares. More generally, we now prove there is only one type of bad $k$-cubes for every $k\geq 2$.

\begin{prop}
Every $k$-cube of the cubulation of $\bar M$ is oriented either coherently, or as $C \times C'$ for some bad square $C$ and a coherently oriented $(k-2)$-cube $C'$.
\end{prop}
\begin{proof}
This holds because the bad ridges in $P^5$ are pairwise disjoint. This can be proved by checking that all the cliques in Figure \ref{quartets:fig} contain at most one red edge.
\end{proof}

We now subdivide the cube complex decomposition of $\bar M$ by cutting each bad $k$-cube $C \times C'$ into four prisms $T \times C'$, where $T$ is one of the four triangles contained in the bad square $C$ as in Figure \ref{subdivision:fig}.

We get a new affine cell decomposition of $\bar M$. All the edges are oriented, all the affine cells have a natural affine function that is compatible with the edges orientations, and these affine functions all match to a circle-valued Morse function $f\colon \bar M \to \matR/\matZ$.

\begin{teo} \label{final:teo}
The ascending and descending links at all the vertices of the affine cell decomposition are collapsible polyhedra. Therefore $f$ is homotopic to a fibration.
\end{teo}
\begin{proof}
There are 3 types of vertices $w$ in the affine cell decomposition:  
\begin{enumerate}
\item the barycenters of some polytope $\bar P_v^5$,
\item the barycenters of some 4-cube facet of $\bar P_v^5$, that lie in $\partial \bar M^5$,
\item the barycenters of the bad squares.
\end{enumerate}

The vertices of type (1) and (2) are those of the original cubulation, while those of type (3) were introduced in the subdivision into prisms. We analyse the three cases separately. 
The collapsibility of the ascending and descending links of the vertices of type (1) and (2) is carefully deduced from Propositions \ref{balanced:collapsible:prop} and \ref{vertices:coll:prop}, while for those of type (3) it follows essentially from Figure \ref{P3_states:fig}. Here are the details.

We start with (1).
Let $w$ be the barycenter of $\bar P_v^5$. Its link $\lk(w)$ in the affine cell decomposition is the boundary of the Gosset polytope $1_{21}$, with a couple of further subdivisions: 
\begin{enumerate}
\item[(i)] the 4-octahedra dual to the ideal vertices of $P^5$ are each subdivided into 16 simplexes by adding a central vertex (which indicates an edge pointing towards a vertex of type (2));
\item[(ii)] one additional vertex is added at the barycenter of each edge dual to a bad ridge of $P_v^5$, and all the adjacent simplexes are subdivided in two (the additional vertex indicates an edge pointing towards a vertex of type (3)).
\end{enumerate}

Recall from Proposition \ref{balanced:collapsible:prop} that the ascending and descending links of the state $s_v$ are some collapsible subcomplexes of the complex $K$ obtained from the boundary of $1_{21}$ by removing its 4-octahedra. 
Note that $\lk(w)$ is obtained from $K$ by adding and subdividing some simplexes (these are the modifications (i) and (ii) described above). The ascending and descending links of $w$ in $\lk(w)$ are obtained from those of $s_v$ by the following corresponding modifications:
\begin{enumerate}
\item[(i)] the central vertex of each 4-octahedron has status O because the edges of the cubulation that intersect the boundary in one endpoint are directed towards it by construction; the central vertex is hence added to the ascending link, together with all the simplexes containing it spanned by vertices with status O; by Proposition \ref{vertices:coll:prop}, the link of this new central vertex in the ascending link is collapsible, so this operation is an expansion (the new ascending link collapses onto the old one); 
\item[(ii)] some simplexes are subdivided, but this does not change the piecewise-linear homomorphism type of the ascending and descending links.
\end{enumerate}
The ascending and descending links of $s_v$ are collapsible by Proposition \ref{balanced:collapsible:prop}, so the resulting ascending and descending links at $w$ also are, and we are done. 

We turn to (2). Let $w \in \partial \bar M^5$ be the barycenter of a boundary $4$-cube facet of some polytope $\bar P^5_v$, which corresponds to an ideal vertex of $P^5$. The link $\lk(w)$ of $w$ in the decomposition is a cone over a polytope that is a 4-octahedron with some further subdivisions at the edges dual to the bad ridges. The parts of the ascending and descending links contained in the subdivided 4-octahedron are collapsible by Proposition \ref{vertices:coll:prop}. The vertex of the cone has status I, so the descending link is further modified by some expansion. Therefore the ascending and descending links at $w$ are collapsible.

Finally, we examine the case (3). Let $w$ be the barycenter of some bad square. This can lie either in the interior or in the boundary of $\bar M^5$. If it lies in the interior, it is dual to a 3-stratum that projects to a bad ridge of $P^5$. The link $\lk(w)$ of $w$ in the decomposition is the join of a square (dual to the bad square) and a triangular prism (dual to the ridge). The square has vertices with status I,O,I,O, as one deduces from Figure \ref{squares:fig}-(center). The triangular prism has the status inherited from the bad ridge. 

By looking carefully at Figure \ref{quartets:fig} we deduce that a bad ridge that is the intersection of two facets with colours $t, t+4$ is also incident (along its 6 triangular faces) to facets of all the other 6 colours, whose stati can be as in Figure \ref{P3_states:fig}. For instance, the 6 facets one obtains from the top-left bad ridge incident to the facets coloured with 1 and 5 are shown in Figure \ref{join:fig}. 
Since the ascending and descending links in Figure \ref{P3_states:fig} are collapsible, they remain so after the join with two vertices with status I (or O).

\begin{figure}
 \begin{center}
  \includegraphics[width = 7 cm]{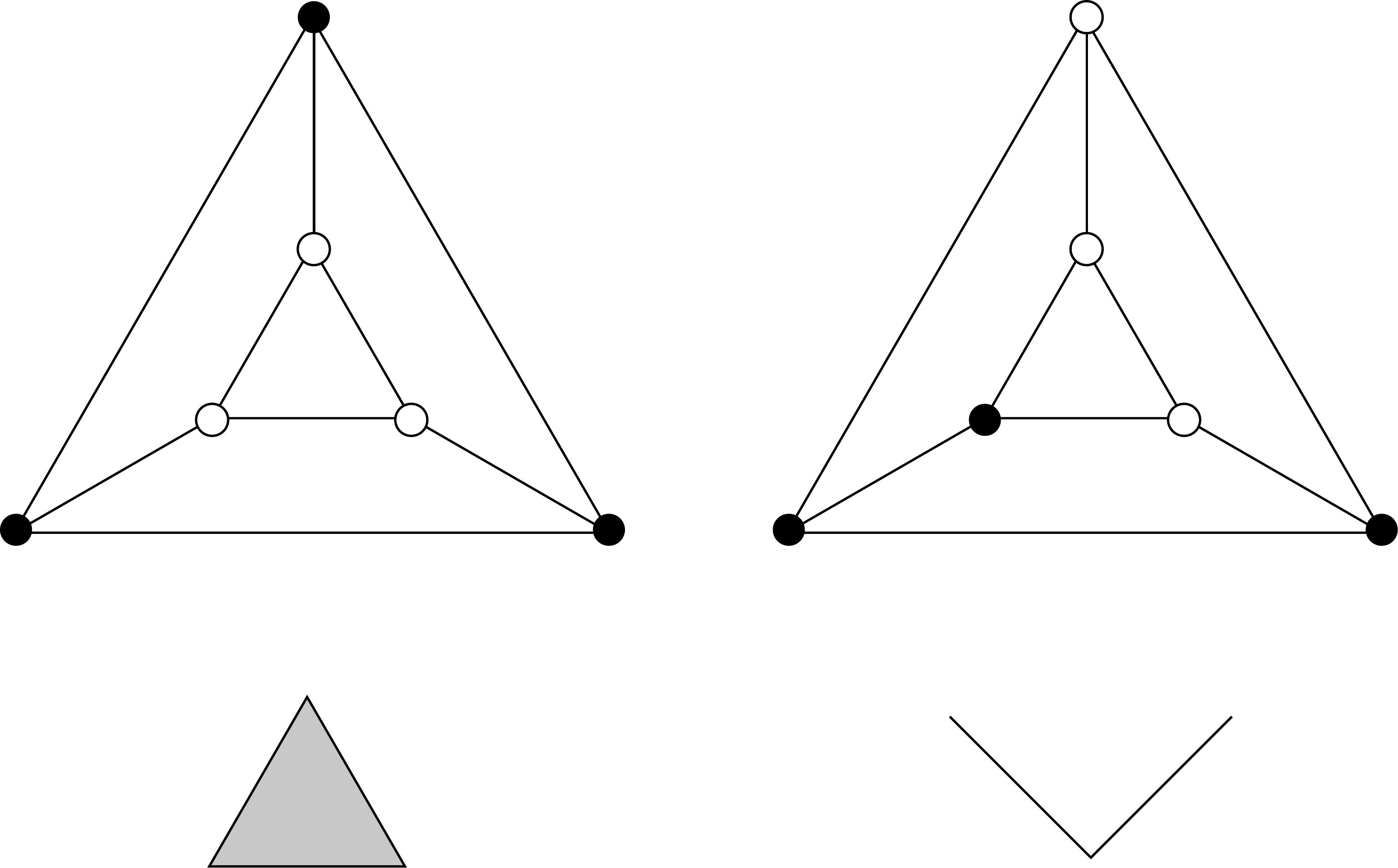}
 \end{center}
 \nota{The possible states of the prisms dual to the bad ridges. In both cases the ascending and descending links are collapsible (a triangle or two edges connected by an endpoint).}
 \label{P3_states:fig}
\end{figure}

\begin{figure}
 \begin{center}
  \includegraphics[width = 5 cm]{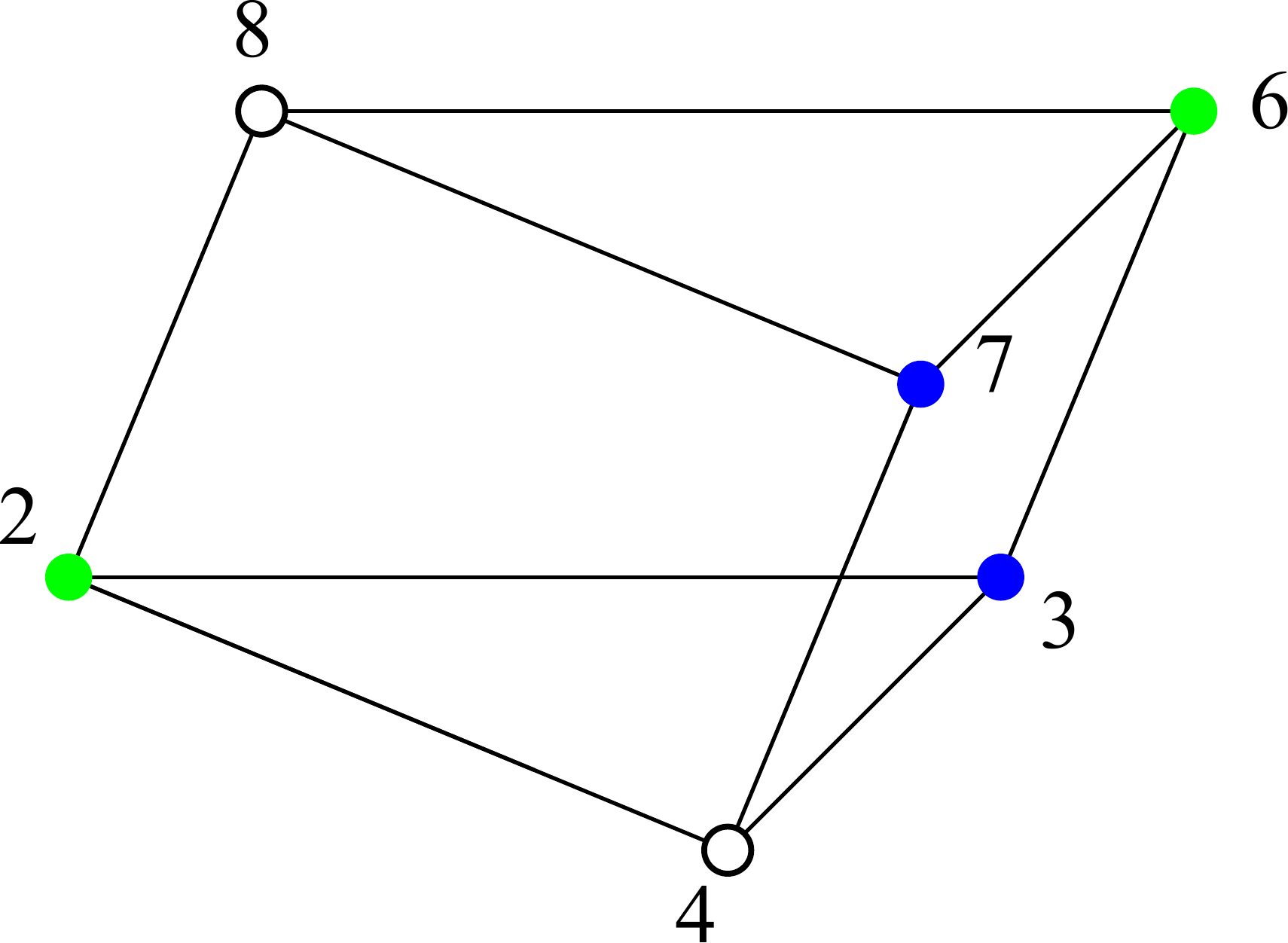}
 \end{center}
 \nota{The 6 other facets incident to the top-left bad ridge adjacent to two facets with colour 1 and 5.}
 \label{join:fig}
\end{figure}

If the bad square lies in $\partial \bar M^5$, the link of $w$ is the cone over the join of the square with vertices I,O,I,O and another square, which is one of the three square faces of the prism shown in Figure \ref{P3_states:fig}. In all the cases shown in the figure the square face contains contractible ascending and descending links and hence we conclude as above.
\end{proof}

The proof of Theorem \ref{main:teo} is now complete.

\subsection{The fibrations on the boundary 4-tori} 
We have constructed a fibration $f\colon \bar M^5 \to S^1$. This of course restricts to a fibration on each of the 40 boundary 4-tori, that we now analyse. 

The decomposition of $\bar M^5$ into 5-dimensional polytopes induce a decomposition of $\partial \bar M^5$ into 4-cubes. Remember that there are 8 large cusps and 32 small cusps. Correspondingly we have 8 large boundary 4-tori in $\partial \bar M^5$ that decompose each as $4 \times 4 \times 4 \times 4$ hypercubes and 32 small boundary 4-tori that decompose each as $2 \times 2 \times 2 \times 2$ hypercubes. 
We identify correspondingly each boundary 4-torus with $\matR^4/4\matZ^4$ and $\matR^4/2\matZ^4$ in the obvious way.

\begin{prop} \label{tori:prop}
Up to permuting some coordinate axis and/or reversing their orientations,
the fibration on each large 4-torus is isotopic to $f(x_1,x_2,x_3,x_4) = x_1$, while that on each small 4-torus is $f(x_1,x_2,x_3,x_4) = x_1+x_2+x_3+x_4$. 
\end{prop}
\begin{proof}
We first note that the dual cube complex of the decompositions of the 4-tori into 4-cubes is isomorphic to the original decomposition (the tessellation of Euclidean space into cubes is self-dual).

We first analyse the small boundary 4-tori. Each such is constructed from a 4-cube equipped with its unique 4-colouring. The fibration on it is in turn induced by a state and a set of moves that are the restrictions of those $s,\calS$ of $\bar M^5$. By analysing these restrictions we find that the state is a balanced state (opposite facets have opposite stati), and the set of moves is the set of opposite facets (the same partition of the colouring). With this configuration it is easy to prove that the fibration on the 4-torus is simply the diagonal fibration $f(x_1,x_2,x_3,x_4) = x_1+x_2+x_3+x_4$. The fibration is already smooth by construction, no isotopy is needed.

The situation on the large boundary 4-tori is quite different and more complicated because these are adjacent to the bad ridges. Each large boundary 4-torus is constructed from a 4-cube equipped with its unique 8-colouring. The fibration is induced by a state and a set of moves on the 4-cube that are the restrictions of those $s, \calS$ of $\bar M^5$, and are as follows.

The set of moves consists of 4 pairs of facets like in the previous case, with one key difference: one pair consists of opposite facets, while the other three pairs consist of adjacent facets. The state has opposite stati on the first pair, and coinciding stati on each of the other three pairs. The bad 2-squares that separate the pairs of adjacent facets are those contained in the bad ridges of $P^5$. 

The resulting orientations on the edges of the dual $4\times 4 \times 4 \times 4$ cube complex are as follows. On the coordinate direction that is dual to the facets of the first pair, all the edges are just oriented coherently towards the same direction. We call it the \emph{preferred direction}. The orientations of the other edges are also invariants under integral translations along the preferred direction. Therefore to understand the orientation of all the edges it suffices to study any big 3-cube $4\times 4 \times 4$ that is orthogonal to the preferred direction. 

The orientations on such a big 3-cube are as shown in Figure \ref{cube:fig}. The figure shows the bad squares and how we decomposed them into triangles. The big 3-cube decomposes into $4\times 4 \times 4$ cubes, that are of two kinds: the \emph{good} ones that do not contain any bad square, and the \emph{bad} ones that contain a pair of opposite bad squares, and were hence decomposed into 4 prisms. The positions of the 48 bad squares and the orientations of the edges that are hidden in the figure can be deduced easily from what is shown. 

The decomposition of the 4-dimensional $4\times 4 \times 4 \times 4$ large boundary 4-torus is the one shown in Figure \ref{cube:fig} multiplied by $[0,1]$, then repeated and juxtaposed 4 times. A fiber of the fibration is a complicated 3-dimensional object, union of various 3-dimensional polyhedra attached along their faces by forming varying angles. However, since there is a preferred direction where all the arrows point in the same way, every such polyhedron will be transverse to it. Every fiber of $f$ is the graph of a piecewise-linear function from a 3-torus orthogonal to the direction to the $S^1$ parallel to the direction. Therefore the fibration is isotopic to the projection $f(x_1,x_2,x_3,x_4) = x_1$, if we suppose that the first coordinate represents the preferred direction.
\end{proof}

\begin{figure}
 \begin{center}
  \includegraphics[width = 7 cm]{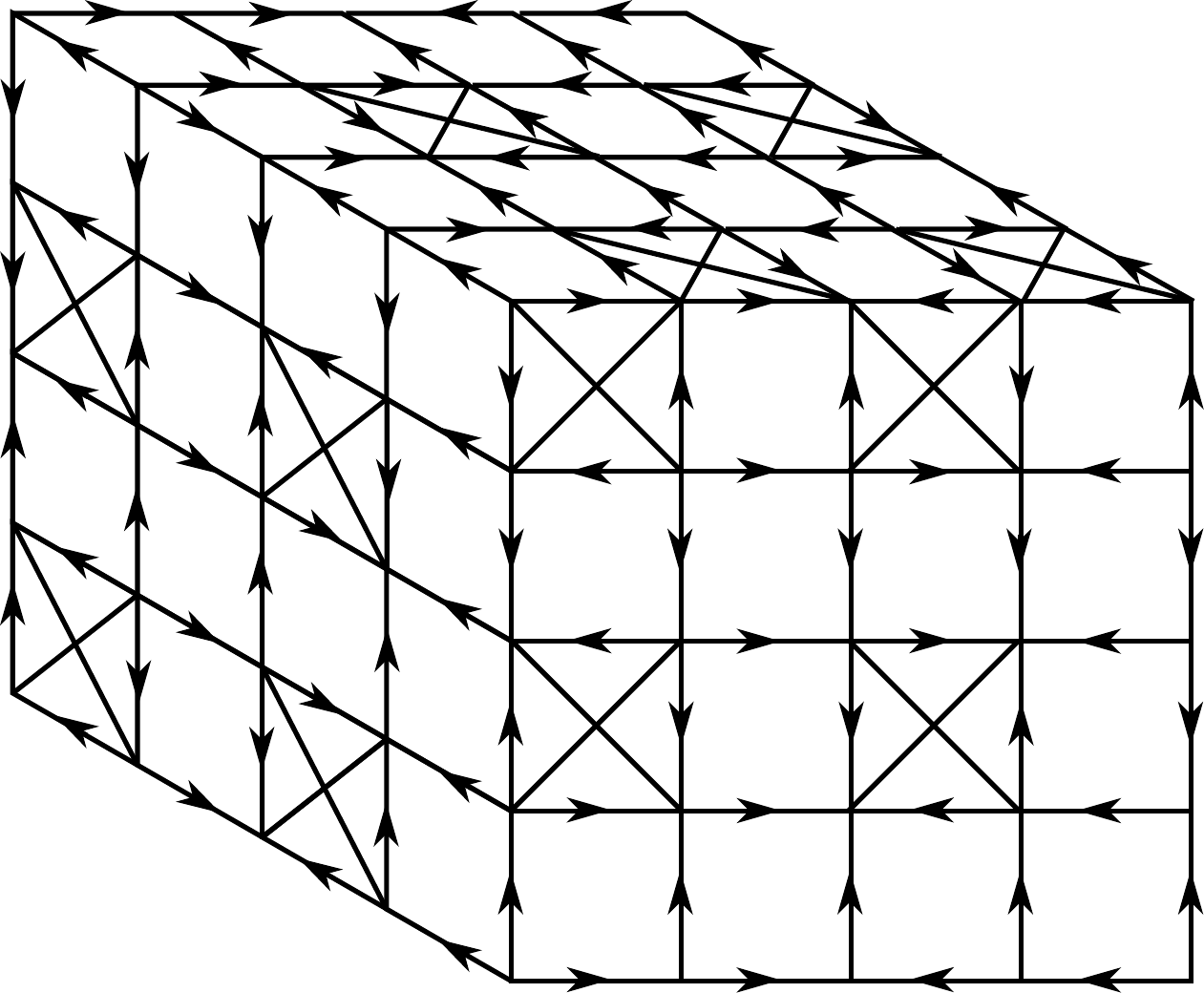}
 \end{center}
 \nota{The orientation on the edges on the $4\times 4 \times 4$ cubes orthogonal to the preferred direction. The bad squares are subdivided into four triangles as shown here.}
 \label{cube:fig}
\end{figure}

\subsection{A totally geodesic fibering 3-dimensional submanifold} \label{3:subsection}
We have constructed a fibration $f\colon M^5 \to S^1$. We now describe some interesting geodesic 3-dimensional submanifolds.

Let $F^k \subset P^5$ be a $k$-face with $1 \leq k \leq 4$. Its counterimage in $M^5$ consists of a (possibly disconnected) totally geodesic $k$-dimensional hyperbolic submanifold that is easily described. The face $F^k$ inherits a colouring from that of $P^5$, where every facet $G^{k-1}$ of $F^k$ takes the colour of the unique facet of $P^5$ that intersects $F^k$ in $G^{k-1}$. With the technique already explained above, the colouring defines a hyperbolic $k$-manifold $M^k$ that orbifold-covers $F^k$. The counterimage $\pi^{-1}(F^k)$ of $F^k$ in $M^5$ consists of $2^c$ copies of $M^k$, where $c$ is the number of colours in the palette $\{1,\ldots,8\}$ that are \emph{not} assigned to any facet of $P^5$ incident to $F^k$.

This yields plenty of interesting geodesic submanifolds in $M^5$. Every two-dimensional face $F^2$ in $P^5$ is a triangle with one real right-angled vertex and two ideal vertices. Depending on whether the triangle inherits 2 or 3 colours from those of $P^5$, its counterimage consists of many geodesic hyperbolic punctured spheres in $M^5$, with 3 or 4 punctures respectively.

Every three-dimensional face $F^3$ in $P^5$ is a bipyramid with three ideal vertices and two real vertices. This is the right-angled polyhedron $P^3$ already considered for instance in \cite{RT4}. Its counterimage consists of a number of cusped hyperbolic 3-manifolds totally geodesically embedded in $M^5$.

We are particularly interested in the case where $F^3$ is a bad ridge. In that case we have already seen during the proof of Theorem \ref{final:teo} that the six faces of $F^3$ inherit six different colours (so $F^3$ is adjacent to 8 facets with 8 distinct colours overall) 
and hence the counterimage of $F^3$ is a connected totally geodesic hyperbolic 3-manifold $M^3\subset M^5$ tessellated into $2^6$ copies of $F^3$. The interesting point here is that all the states of these $2^6$ copies inherited from the states of $P^5_v$ are as shown in Figure \ref{P3_states:fig}. All the ascending and descending links are collapsible, and therefore we deduce the following.

\begin{prop}
The fibration $f\colon M^5 \to S^1$ restricts to a fibration on each totally geodesic $M^3 \subset M^5$ that is the counterimage of a bad ridge of $P^5$. 
\end{prop}

There are eight such disjoint geodesically embedded $M^3$, one above each bad ridge of $P^5$. By analysing $M^3$ we discover that it has 12 toric cusps, all contained in the large cusps of $M^5$.

\subsection{Determination of the fiber} \label{determination:fiber:subsection}
The fiber of the fibration $f\colon \bar M^5 \to S^1$ is a compact 4-manifold $\bar F^4$ with boundary, and it is of course of primary interest to understand the topology of $\bar F^4$ and the monodromy of the fibration. We denote by $F^4$ the interior of $\bar F^4$, that is the fiber of the fibration $f\colon M^5 \to S^1$. 

A partially ideal triangulation of $F^4$ can be built quite easily as follows. In each $P^5_v$, let $\Sigma^3_v \subset \partial P^5_v$ be the union of all the 36 ridges that separate two facets with opposite status (these ridges are never bad). Each ridge is a double pyramid with three ideal vertices and two finite vertices. The piecewise-linear object $\Sigma^3_v$ is a punctured 3-sphere inside the punctured 4-sphere $\partial P^5_v$, with the punctures being at the ideal vertices. The punctured 3-sphere $\Sigma^3_v$ separates the two punctured 4-discs in $\partial P^5_v$ that consist of all the facets with status O or I, respectively.

Let $C(\Sigma^3_v)\subset P^5_v$ be the cone over $\Sigma^3_v$ centered at the barycenter of $P^5_v$. This is a 4-disc with some points removed from the boundary (again, the ideal vertices). One can verify that, up to an isotopy, we have
\begin{equation} \label{F4:eqn}
F^4 = \bigcup_{v\in\matZ_2^8} C(\Sigma^3_v).
\end{equation}

We have thus described the fiber as a concrete piecewise-linear object $F^4$ contained in $M^5$. Of course $F^4$ is very far from being geodesic: it decomposes into some 4-dimensional polytopes (the cones over the ridges) that form various angles at their 3-dimensional common facets.

Each ridge is a double pyramid that can be triangulated canonically into two tetrahedra; each such tetrahedron has 3 ideal vertices and one real vertex. Accordingly each $C(\Sigma^3_v)$ can be triangulated into 72 distinct 4-simplexes, where each 4-simplex has 3 ideal vertices and 2 real ones. Summing up, the manifold $F^4$ is triangulated into $72 \times 2^8=18432$ distinct 4-simplexes with both ideal and finite vertices. This triangulation is too large to be studied, and this motivated us to look for some smaller examples. We do this in the next section.

\begin{rem} \label{unwrap:rem}
As is typical with these constructions, the fiber $F^4$ has two connected components, obtained by taking in (\ref{F4:eqn}) the vectors $v\in \matZ_2^8$ with even or odd sum $v_1+\cdots + v_8$. The fibration $f\colon M^5 \to S^1$ that we have constructed has disconnected fibers, and the image of $f_*\colon\pi_1(M^5) \to \pi_1(S^1) = \matZ$ is $2\matZ$. We will henceforth substitute $f$ with its lift along the degree-2 covering $S^1 \to S^1$ to get a fibration with connected fibers. We still name the resulting fibration as $f\colon M^5 \to S^1$ for simplicity. 
\end{rem}

\section{A smaller example} \label{quotient:section}
In the previous section we have constructed a fibration $f\colon M^5 \to S^1$.
We now build a smaller example $N^5$ that is commensurable with $M^5$. The smaller example $N^5$ turns out to be the smallest hyperbolic 5-manifold known, constructed by Ratcliffe and Tschantz in \cite{RT5}.
For such an object we can describe the fiber explicitly and study it using Regina \cite{regina}.

\subsection{The abelian cover} 
Let $\tilde M^5 \to M^5$ be the abelian cover obtained by unwrapping completely the fibration $f \colon M^5 \to S^1$. Formally, this is the cover determined by the kernel of $f_*\colon  \pi_1(M^5) \to \pi_1(S^1) = \matZ$. 
The fibration $f\colon M^5 \to S^1$ lifts to a fibration $\tilde f \colon \tilde M^5 \to \matR$ with the same fiber $F^4$ as the original $f$. The hyperbolic manifold $\tilde M^5$ has infinite volume, it is geometrically infinite and diffeomorphic to $F^4 \times \matR$.

The tessellation of $M^5$ into polytopes $P_v^5$ with $v\in \matZ_2^8$ lifts naturally to a tessellation of $\tilde M^5$ that can be described explicitly. The manifold $\tilde M^5$ decomposes into infinitely many copies of $P^5$ indexed as
$$P_{v,l}^5$$
where $v \in \matZ_2^8$ and $l \in \matZ$ vary with the condition that
$v_1+ \cdots + v_8 + l $ should be an even integer. (This condition is connected to Remark \ref{unwrap:rem}: without it, we would get two identical copies of $\tilde M^5$). The polytope $P_{v,l}^5$ inherits the colouring of $P^5$ and the state of $P_v^5$.
The facet $F$ of $P_{v,l}^5$ is identified with the identity map to the facet $F$ of $P_{v+e_i,l \pm 1}^5$, where $i$ is the colour of $F$ and the sign $+1$ or $-1$ depends on whether the status of $F$ in $P_{v,l}^5$ is O or I.

The covering $\tilde M^5 \to M^5$ is simply the forgetful map $P_{v,l}^5 \to P_v^5$. The monodromy of the covering is the deck transformation 
$$\tau \colon \tilde M^5 \to \tilde M^5$$ 
that sends $P_{v,l}^5$ to $P_{v,l+2}^5$ via the identity map. Of course we have
$$M^5 = \tilde M^5 /{\langle \tau \rangle}.$$

\subsection{The Ratcliffe--Tschantz hyperbolic 5-manifold $N^5$}
We now quotient $\tilde M^5$ to obtain a much smaller hyperbolic manifold $N^5$, that we find \emph{a posteriori} to be isometric to the smallest hyperbolic 5-manifold known, discovered by Ratcliffe and Tschantz in \cite{RT5}.

Recall from Proposition \ref{16:prop} that $R_{16}$ is a group of isometries of $P^5$ that preserve the colouring of $P^5$ and acts freely and transitively on the balanced states of $P^5$. This implies that for every pair of polytopes $P^5_{v,l}$ and $P^5_{v',l'}$ of the tessellation of $\tilde M^5$ there is a unique state-preserving isometry $\varphi\in R_{16}$ that sends the first to the second (recall that these polytopes inherit the states of $P^5_v$ and $P^5_{v'}$). 
This is a very convenient property. We now show that this isometry extends to an isometry of $\tilde M^5$, and that all the isometries of $\tilde M^5$ constructed in this fashion form a group.

Every isometry $\varphi\in R_{16}$ acts on the set $\{1, \ldots, 8\}$ of colourings, and hence linearly on $\matZ_2^8$ by sending $e_i$ to $e_{\varphi(i)}$. 

\begin{lemma} \label{Phi:lemma}
Let $\varphi \in R_{16}$ be a state-preserving isometry
$\varphi \colon P^5_{v,l} \to P^5_{v',l'}$. It extends to a unique isometry $\Phi\colon \tilde M^5 \to \tilde M^5$, which sends $P^5_{v+w,l+j}$ to $P^5_{v'+\varphi(w),l'+j}$ via the states-preserving isometry $\varphi$, for every $w \in \matZ_2^8$ and $j\in \matZ$ with $w_1+\cdots+w_8+j$ even. We have $\varphi \in Q_8$ if and only if $l-l'$ is even. 

All the isometries $\Phi$ of this kind form a group $\Gamma$ that acts freely and transitively on the polyhedra $P^5_{v,l}$ of the tessellation.
\end{lemma}
\begin{proof} 
We show that $\Phi$ is well-defined. By definition $\Phi$ sends the polyhedron $P^5_{v+w,l+j}$ to $P^5_{v'+\varphi(w),l'+j}$ via $\varphi$, for every $w \in \matZ_2^8$ and $j\in \matZ$ with $w_1+\cdots+w_8+j$ even. It is immediate to check that $\Phi$ permutes the polyhedra of the tessellation of $\tilde M$. It is an isometry when restricted to each of them, and it remains to prove that it is well-defined on their intersections. 

Consider a facet $F$ that separates two polyhedra $P^5_{v+w,l+j}$ and $P^5_{v+w+e_i,l+j+1}$. The colour of $F$ is $i$. The map $\Psi$ sends the two polyhedra adjacent to $F$ respectively to $P^5_{v'+\varphi(w),l'+j}$ and $P^5_{v'+\varphi(w+e_i),l'+j+1} = P^5_{v'+\varphi(w) + e_{\varphi(i)}, l'+j+1}$, and it sends $F$ to the facet $\varphi(F)$ of colour $\varphi(i)$ contained in both of them. Therefore $\Phi$ is well-defined.

The fact that $\varphi \in Q_8$ if and only if $l-l'$ is even follows from the discussion after Proposition \ref{16:prop}, where we say that $\varphi \in Q_8$ if and only if the number of ``bars'' is divisible by four.
\end{proof}

We have just defined a group $\Gamma$ of isometries of $\tilde M$ that acts freely and transitively on the polytopes $P^5_{v,l}$ of the tessellation. The action on $\tilde M$ is not free, because the states-preserving map $$\id \colon P^5_{0,0} \longrightarrow P^5_{e_1+e_5,0}$$
extends to an isometry $\Phi$ of $\tilde M^5$ which fixes the non-empty intersection
$P^5_{0,0} \cap P^5_{e_1+e_5,0}$, that consists of two bad ridges separating the facets with colour 1 and 5. The quotient $\tilde M^5/\Gamma$ is an interesting orbifold (not a manifold) tessellated into a single copy of $P^5$.

We now describe a torsion-free index two subgroup $\Gamma_0 \triangleleft \Gamma$. We consider the following isometries in $\Gamma$:
\begin{enumerate}
\item those that send $P_{0,0}^5$ to $P_{v,0}^5$ with $v_1+ \cdots + v_4$ even;
\item the isometry that sends $P_{0,0}^5$ to $P_{e_1,1}^5$.
\end{enumerate}

Let $\Gamma_0<\Gamma$ be the subgroup generated by these (finitely many) isometries. 

\begin{teo}
The group $\Gamma_0$ has index two in $\Gamma$ and acts freely on $\tilde M^5$. The quotient $N^5 = \tilde M^5/\Gamma_0$ is a cusped hyperbolic 5-manifold tessellated in two copies of $P^5$, isometric to the manifold discovered by Ratcliffe and Tschantz in \cite{RT5}.
\end{teo}
\begin{proof}
Let the \emph{parity} of a vector $w\in \matZ_2^8$ be the parity of $w_1+w_2+w_3+w_4$.
By using the explicit description of $\Phi$ of Lemma \ref{Phi:lemma} we can check that  the isometries of type (1) send $P^5_{w,l}$ to some $P^5_{w',l}$ such that $w$ and $w'$ have the same parity, while the isometry (2) sends $P^5_{w,l}$ to some $P^5_{w',l+1}$ such that $w$ and $w'$ have the same parity $\Leftrightarrow$ $l$ is odd. To prove this, we use that for (1) $\varphi \in Q_8$ acts on the colourings by preserving the two subsets $\{1,\ldots, 4\}$ and $\{5,\ldots, 8\}$, while for (2) $\varphi \not \in Q_8$ exchanges these two subsets. So for (1) we have $w' = v + \varphi(w)$ and hence
$$w_1+ \cdots +w_4 + v_1+\cdots +v_4+\varphi(w)_1+\cdots + \varphi(w)_4
= 2w_1 + \cdots + 2w_4 + v_1 + \cdots + v_4$$
is even, while for (2) we have $w' = e_1+\varphi(w)$ and hence
$$w_1+ \cdots +w_4 + 1+\varphi(w)_1+\cdots + \varphi(w)_4
= w_1 + \cdots + w_8 + 1 $$
is even $\Leftrightarrow$ $l$ is odd. From these facts we deduce easily that $P_{v,l}$ and $P_{w,l}$ lie in the same orbit with respect to $\Gamma_0$ if and only if $v$ and $w$ have the same parity. This shows that $\Gamma_0< \Gamma$ has index two and that $P_{0,0}^5$ and $P_{e_1+e_5,0}^5$ are representatives of the two orbits. Hence the quotient $N^5$ is tessellated into two copies of $P^5$.

We prove that the action is free. Let $\Phi \in \Gamma_0$ send $P^5_{v,l}$ to $P^5_{v',l'}$ via some isometry $\varphi$. If $l'\neq l$, then $\Phi$ has infinite order and therefore it has no fixed points (the action of $\Gamma$ is properly discontinuous). If $l'=l$ and $\varphi \neq \id$, then $\varphi \in Q_8$ is an isometry of $P^5$ that fixes the vertical axis $\{0\}\times \matR \subset \matH \times \matR$, which does not intersect any facet of $P^5$; this implies that $\Phi$ does not have any fixed point in $\tilde M^5$. 

If $l'=l$, $\varphi = \id$, and $v\neq v'$, then $\varphi = \id$ implies that $P_{v,l}$ and $P_{v',l}$ have the same status, which holds precisely if $v_i + v_{i+4} + v_i' + v_{i+4}'$ is even for all $i\in \{1,\ldots, 4\}$. Moreover, by hypothesis also $v_1+\cdots + v_4+v_1'+\cdots + v_4'$ is even: from these conditions we deduce that there are at least two distinct numbers $i,j \in \{1,\ldots, 4\}$ such that $v$ and $v'$ differ in the coordinates $i,j,i+4,j+4$. This implies that $P^5_{v,l} \cap P^5_{v',l} = \emptyset$ (because distinct bad ridges of $P^5$ are disjoint) and hence $\Phi$ has no fixed points.

The fact that the resulting manifold $N^5$ is isometric to the one found by Ratcliffe and Tschantz was then found via computer, by constructing triangulations of both manifolds and then proving that they are combinatorially isomorphic. The code is available in \cite{code}.
\end{proof}

\begin{table}
\begin{center}
Polytope A
\begin{tabular}{ccccc}
Facet        &           Status & To Polytope & To Facet   &             Permutation \\
$(-1, -1, -1, -1,+1)$  &   In  &    B       &    $(-1,+1,+1,+1, -1)$  &     13245 \\       
$(-1, -1, -1,+1, -1)$   &  In   &   A     &      $(+1,+1, -1, -1,+1)$   &    42315  \\    
$(-1, -1,+1, -1, -1)$   &  In    &  A   &        $(+1, -1,+1, -1,+1)$    &   24135  \\  
$(-1, -1,+1,+1,+1)$   &    In   &   B   &        $(+1,+1,+1, -1, -1)$    &   42315 \\      
$(-1,+1, -1, -1, -1)$   & In   &   A   &        $(+1, -1, -1,+1,+1)$  &     31425  \\     
$(-1,+1, -1,+1,+1)$   & In   &   B   &        $(+1,+1, -1,+1, -1)$  &     31425   \\    
$(-1,+1,+1, -1,+1)$   &  In    &  B     &      $(+1, -1,+1,+1, -1)$   &    24135   \\    
$(-1,+1,+1,+1, -1)$    &  Out &    B   &        $(-1, -1, -1, -1,+1)$  &   13245   \\    
$(+1, -1, -1, -1, -1)$   & In   &   A   &        $(+1,+1,+1,+1,+1)$   &      13245   \\    
$(+1, -1, -1,+1,+1)$    & Out  &   A   &        $(-1,+1, -1, -1, -1)$  &   24135  \\     
$(+1, -1,+1, -1,+1)$   &   Out &    A    &       $(-1, -1,+1, -1, -1)$  &   31425  \\     
$(+1, -1,+1,+1, -1)$    &  Out  &   B  &         $(-1,+1,+1, -1,+1)$  &     31425  \\     
$(+1,+1, -1, -1,+1)$    &  Out  &   A  &         $(-1, -1, -1,+1, -1)$ &    42315  \\     
$(+1,+1, -1,+1, -1)$   &  Out  &   B   &        $(-1,+1, -1,+1,+1)$   &    24135   \\    
$(+1,+1,+1, -1, -1)$  &    Out  &   B  &         $(-1, -1,+1,+1,+1)$  &     42315   \\
$(+1,+1,+1,+1,+1)$   &     Out  &   A    &       $(+1, -1, -1, -1, -1)$ &    13245       
\end{tabular}    
\vspace{.3 cm}

Polytope B  
\begin{tabular}{ccccc}    
Facet         &           Status &  To Polytope &  To Facet      &           Permutation \\
$(-1, -1, -1, -1,+1)$  &    In  &     A    &        $(-1,+1,+1,+1, -1)$  &      13245    \\   
$(-1, -1, -1,+1, -1)$ &     In   &    B    &        $(+1,+1, -1, -1,+1)$  &      42315     \\   
$(-1, -1,+1, -1, -1)$ &     In   &    B    &        $(+1, -1,+1, -1,+1)$   &     24135     \\   
$(-1, -1,+1,+1,+1)$ &       In  &     A    &        $(+1,+1,+1, -1, -1)$   &     42315     \\   
$(-1,+1, -1, -1, -1)$  &    In  &     B    &        $(+1, -1, -1,+1,+1)$  &      31425    \\    
$(-1,+1, -1,+1,+1)$  &      In  &     A    &        $(+1,+1, -1,+1, -1)$  &      31425    \\    
$(-1,+1,+1, -1,+1)$  &      In   &    A    &        $(+1, -1,+1,+1, -1)$   &     24135    \\    
$(-1,+1,+1,+1, -1)$  &      Out  &    A    &        $(-1, -1, -1, -1,+1)$ &     13245  \\      
$(+1, -1, -1, -1, -1)$  &    In   &    B    &        $(+1,+1,+1,+1,+1)$    &      13245    \\    
$(+1, -1, -1,+1,+1)$  &      Out  &    B   &         $(-1,+1, -1, -1, -1)$  &    24135  \\      
$(+1, -1,+1, -1,+1)$  &      Out  &    B   &         $(-1, -1,+1, -1, -1)$ &     31425   \\     
$(+1, -1,+1,+1, -1)$  &      Out  &    A   &         $(-1,+1,+1, -1,+1)$   &     31425    \\    
$(+1,+1, -1, -1,+1)$ &       Out  &    B  &          $(-1, -1, -1,+1, -1)$ &     42315   \\     
$(+1,+1, -1,+1, -1)$  &      Out &     A  &          $(-1,+1, -1,+1,+1)$   &     24135    \\    
$(+1,+1,+1, -1, -1)$  &      Out &     A  &          $(-1, -1,+1,+1,+1)$   &     42315    \\    
$(+1,+1,+1,+1,+1)$   &       Out &     B   &         $(+1, -1, -1, -1, -1)$  &    13245       
\end{tabular}
\vspace{.5 cm}
\nota{The Ratcliffe -- Tschantz manifold $N^5$ is obtained by pairing the facets of two copies of $P^5$ as indicated here. The states are also shown.}
\label{N5:table}
\end{center}
\end{table}

In Table \ref{N5:table} we show how to obtain $N^5$ as an isometric facet-pairing of two copies of $P^5$. The second column in the table shows the status I/O of each facet, and the last column indicates the permutation of the 5 axis of $\matR^5$ induced by the isometric pairing: the permutation is enough to recover the isometry that glues the two facets. It is clear from the table that if we exchange the two polytopes via the identity map we get a status-preserving isometry of $N^5$, yielding the fibering orbifold $\tilde M/\Gamma$. 
The pairing is combinatorially equivalent to the one exhibited in \cite{RT5}, as we checked with a computer.

Some topological information on the manifold $N^5$ was already given in \cite[Section 9]{RT5}. We have
$$H_1(N^5) = \matZ \oplus \matZ_4, \quad H_2(N^5) = \matZ_4^2, \quad
H_3(N^5) = \matZ, \quad H_4(N^5) = \matZ.$$

This is the hyperbolic 5-manifold with the smallest volume $7\zeta(3)/4 = 2.103\ldots $ known. It is orientable and has two cusps, whose sections are the flat 4-manifolds denoted as $D$ and $P$ in \cite{RT5}. Both these flat 4-manifolds fiber over $S^1$ with fiber the Hantsche--Wendt flat 3-manifold $\HW$, the first with trivial monodromy (so $D = \HW \times S^1$), and the second with a non-trivial one. We have
$$H_1(D) = \matZ \oplus \matZ_4^2, \quad H_1(P) = \matZ \oplus \matZ_4.$$

\subsection{The fiber} 
Using Sage we have written the algorithm described in Section \ref{determination:fiber:subsection} to construct the fiber $F^4$ of the fibration $f\colon N^5 \to S^1$. The code is available from \cite{code}. The resulting fiber $F^4$ is triangulated into $2\cdot 72 = 144$ distinct 4-dimensional simplexes, a much more reasonable number than the one found in Section \ref{determination:fiber:subsection}. 





The triangulation of $F^4$ has both ideal and real vertices. We have then asked Regina \cite{regina} to simplify it, and the result is a nice triangulation of $F^4$ with only 40 distinct 4-simplexes. Its isomorphism signature is

\begin{center}
{\tt \small
\hash{OvvvAPMAAwMvzAAQQwvAwLQQQQPAvQQQQQciggkhhikjlnppptuqwrvxutxvCDAFIDIHDEBIGCBEFGJHKNLNNMLKKMMLNLKMAaAa8awb8awbwb8aaaaaaa8aAa8aaa8a8awb8aaaaaAa8awbAawbwbAaAawbwbaaaaaawbaaaawbaaaawbwbaaaaAaaa8aAaAaaaAaaaaawbwb8a8aAaAaaa8a}
}
\end{center}

This is is a nice purely ideal triangulation, with 5 ideal vertices, 17 edges, 78 triangles, 100 tetrahedra, and 40 polychora. Regina also found some triangulations with 36 distinct 4-simplexes, that look however less symmetric than the one described above. One has the following isomorphism signature:

\begin{center}
{\tt \small
\hash{KvLLAAMzLvLLPPvLQQvQQQQQzwQQQQQfcgdgfhjgnnuroouwxpAxwtABDAzBvCAEzvBwxzFEDEDJIHGIHJIHJJ2a2aDaJaaaJaaaaaVbaaaataaaWaWaDa1aTbPb2aPbRa2afayaDagafbRaJaRafata1ayaJayaJayataJbaaEaEaaa2avaga2avaJaJavafbya}
}
\end{center}

Regina tells us that $F^4$ is the interior of a compact manifold $\bar F^4$ with 5 boundary components, each homeomorphic to \HW. This is of course coherent with the fact that the cusp sections of $N^5$ are bundles with \HW\ fibers. The fact that $\bar F^4$ has more boundary components than $\bar N^5$ is of course not a contradiction, since many boundary components of $\bar F^4$ can lie in the same boundary component of $\bar N^5$. Regina computes the integral homology of $F^4$, which is
$$H_1(F^4) = \matZ_4^4,\quad
H_2(F^4) = \matZ^4, \quad
H_3(F^4) = \matZ^4.
$$

We deduce that $\chi (F^4)=1$. This implies in particular that there is no way to obtain a smaller fiber by further quotienting $N^5$ by some isometries.

The manifold $F^4$ is of course aspherical, since $N^5$ is aspherical and covered by $F^4\times \matR$. Regina furnishes a presentation of its fundamental group with 4 generators ${\tt a,b,c,d}$ and 8 relations:
\begin{gather*}
{\tt a^2 b a^2 b^{-1}, \quad a^{-1} b^2 a b^2, \quad c^2 d^{-1} c^2 d, \quad
    c d^2 c^{-1} d^2}, \\
{\tt    a d^{-1} c^{-1} a^{-1} c b d b^{-1}, \quad   
a c a^{-1} b^{-1} c^{-1} b d^{-1} b d b^{-1},} \\
{\tt 
a c^{-1} a c d^{-1} c b^{-1} a^2 c^{-1} d b, \quad
    a c^{-1} d b a b d^{-1} c b a c b a c^{-1}}.
\end{gather*}

\subsection{A totally geodesic fibering 3-dimensional submanifold}
We have noticed in Section \ref{3:subsection} that $M^5$ contains eight disjoint totally geodesic hyperbolic 3-manifolds $M^3$ such that $f|_{M^3}$ is a fibration on each. These manifolds $M^3$ are precisely the counterimages of the 8 bad ridges of $P^5$.

The same phenomenon occurs in the fibration $f\colon N^5 \to S^1$. The bad ridges of the two copies of $P^5$ that form the tessellation of $N^5$ glue to form a unique connected totally geodesic hyperbolic 3-manifold $N^3\subset N^5$, that is covered by the 3-manifold $M^3$ considered above. By direct inspection we find that the fiber of $f\colon N^3 \to S^1$ is a torus with three punctures, naturally contained in the fiber $F^4$. 

\subsection{The monodromy?}
Like every other fibration, the one $f\colon N^5 \to S^1$ that we have constructed here could be fully described topologically by exhibiting the fiber $F^4$ and the monodromy diffeomorphism $\varphi \colon F^4 \to F^4$. Thanks to Regina, we have a quite reasonable description of the fiber $F^4$, but admittedly we do not have yet a description of the monodromy, because it is quite hard to infer it from the combinatorial data that we have.

Nevertheless, after analysing the triangulation furnished by Regina we make the following guess. Let $\Sigma^2$ be a torus with three holes. 

\begin{guess}
The 4-manifold $\bar F^4$ is obtained from one boundary component $\HW$ by attaching first $\Sigma^2 \times D^2$ along an embedding $(\partial\Sigma^2) \times D^2 \hookrightarrow \HW$, and then some handles of dimension $\geq 2$. The monodromy $\varphi$ is the identity on $\HW$, it preserves $\Sigma^2$, and acts on it like a pseudo-Anosov diffeomorphism.
\end{guess}

The image of the embedding of $(\partial \Sigma^2) \times D^2$ in $\HW$ should be a thickened link with 3 components (that is, three disjoint solid tori). The surface $\Sigma^2$ should be the transverse intersection of the geodesic 3-manifold $N^3$ described above and the fiber $\bar F^4$. At the fundamental group level, if this picture were true we would deduce that $\pi_1(\bar F^4)$ is generated by the injectively embedded subgroups $\pi_1(\HW)$ and $\pi_1(\Sigma^2)$, and the monodromy acts on the first like the identity and on the second like a pseudo-Anosov diffeomorphism. That might lead (if it were true!) to a reasonable description of the monodromy.

\section{Non-hyperbolic groups} \label{GGT:section}
We prove here Corollary \ref{main:cor}. Let $M^5$ be the cusped hyperbolic 5-manifold that fibers over $S^1$, with fiber $F^4$, constructed in Section \ref{fibration:section}. (We may suppose after unwrapping that the fiber $F^4$ is connected, see Remark \ref{unwrap:rem}.)

Recall that $M^5$ has 40 toric cusps.
We pick 40 disjoint embedded sections of the cusps and truncate $M^5$ along them to get a compact manifold $\bar M^5$, whose interior is diffeomorphic to $M^5$, and whose boundary consists of 40 Euclidean 4-tori. By Proposition \ref{tori:prop} the restriction $f \colon T^4 \to S^1$ to each boundary Euclidean 4-torus $T^4\subset \partial M^5$ is isotopic to a \emph{geodesic} fibration, that is one where each fiber is a geodesic 3-torus in $T^3$.
We modify the original fibration $f\colon \bar M^5\to S^1$ via this isotopy so that its restriction to each boundary component $T^4$ is geodesic. Let $\bar F^4$ be the fiber of $f\colon \bar M^5 \to S^1$, a compact 4-manifold bounded by 3-tori whose interior is diffeomorphic to $F^4$.

By the residual finiteness of $\pi_1(M^5)$, up to passing to a finite regular cover we may suppose that each 4-torus boundary $T^4$ of $\bar M^5$ has systole larger than $2\pi$. The fibration lifts obviously to the finite cover. We keep denoting the manifold and the fiber by $M^5$ and $F^4$. The number of boundary 4-tori of $\bar M^5$ may now be bigger than 40, but this is not a problem.

We now follow the terminology of \cite{FM}. Let $\hat M^5$ be the space obtained from $\bar M^5$ by filling along the geodesic $3$-tori fibers in all the boundary components. Topologically, we are shrinking every geodesic 3-torus fiber to a point, thus replacing every boundary component of $\bar M^5$ with a circle. The complement of these circles in $\hat M^5$ is homeomorphic to $M^5$. We set $G=\pi_1(\hat M^5)$.

\begin{prop}
The space $\hat M^5$ is aspherical and $G$ is hyperbolic and torsion-free.
\end{prop}
\begin{proof}
This is a \emph{2$\pi$-filling} because each of the shrinked 3-tori has systole larger than $2\pi$ (there are no closed geodesics shorter than $2\pi$ in the ambient 4-torus). Therefore \cite[Theorem 2.7]{FM} implies that $\hat M^5$ has a locally CAT($-\kappa$) complete path metric for some $\kappa > 0$. In particular $\hat M^5$ is aspherical and $G=\pi_1(\hat M^5)$ is hyperbolic and torsion-free.
\end{proof}

The smooth fibration $f\colon \bar M^5 \to S^1$ quotients to a topological fibration $f\colon \hat M^5 \to S^1$ whose fiber is homeomorphic to the space $\hat F^4$ obtained from $\bar F^4$ by coning each of its 3-torus boundary components. At the fundamental groups level we get $f_* \colon G \to \matZ$ and we define
$$H=\pi_1(\hat F^4) = \ker f_*.$$
\begin{prop}
The space $\hat F^4$ is aspherical and hence $H$ is of finite type.
\end{prop}
\begin{proof}
The space $\hat M^5$ is aspherical and covered by $\hat F^4 \times \matR$, which is hence also aspherical, so $\hat F^4$ also is.
\end{proof}

It only remains to prove that $H$ is not hyperbolic.
We note that $\hat M^5$ and $\hat F^4$ are aspherical orientable pseudo-manifolds of dimension 5 and 4. In particular we get the following.
\begin{prop} We have
$H^5(G) = H^5(\hat M^5) = \matZ$ and $H^4(H) = H^4(\hat F^4) = \matZ$. 
\end{prop}
\begin{proof}
We have $M^5 = \hat M^5 \setminus S$ where $S$ is the singular set and consists of circles. The exact sequence of the pair gives $H^5(\hat M^5) = H^5(\hat M^5, S)$. If we denote by $\nu S$ a closed regular neighbourhood of $S$ we have $H^5(\hat M^5, S) = H^5(\hat M^5, \nu S) = H^5(\bar M^5, \partial \bar M^5) = \matZ$ after excising the interior of $\nu S$. The proof for $\hat F^4$ is similar with $S$ consisting of points.
\end{proof}

The crucial point to show that $H$ is not hyperbolic is the following.

\begin{prop}
The outer automorphism group $\Out(H)$ is infinite.
\end{prop}
\begin{proof}
This holds because the monodromy of the fibration has infinite order. To show this, let $\alpha \in \pi_1(\hat M^5)$ be an element whose image $f_*(\alpha)$ generates $\pi_1(S^1)$. Conjugation by $\alpha^k$ is an automorphism of $H$ that is non-trivial in $\Out(H)$ for each $k\neq 0$. Indeed, if by contradiction we had $\alpha^k g \alpha^{-k} = \beta g \beta^{-1}$ for some $\beta \in H$ and all $g\in H$, then $\beta^{-1}\alpha^k \neq e$ would commute with every element of $H$. Since $G$ is hyperbolic, the centraliser of $\beta^{-1}\alpha^k$ is infinite cyclic \cite[Corollary III.$\Gamma$.3.10]{BH}, and hence cyclic since $G$ is torsion-free (every virtually cyclic torsion free group is cyclic). Therefore $H$ is infinite cyclic, a contradiction since $H^4(H) = \matZ$.
\end{proof}

Finally, suppose by contradiction that $H$ is hyperbolic. Since $\Out(H)$ is infinite, by Rips' theory \cite[Corollary 1.3]{BF} the group $H$ splits over a cyclic subgroup. However, this is impossible.

\begin{prop}
The group $H$ cannot split over a cyclic subgroup.
\end{prop}
\begin{proof}
Suppose that $H= A*_\matZ B$ (the cases where the cyclic group is trivial and/or we have a HNN extension are similar). We get an exact sequence
$$\cdots \longrightarrow H^5(\matZ) \longrightarrow H^4(H) \longrightarrow H^4(A) \oplus H^4(B) \longrightarrow H^4(\matZ) \longrightarrow \cdots$$
and hence $H^4(A) \oplus H^4(B) = H^4(H) = \matZ$.
Both $A$ and $B$ have infinite index in $H$, so they are the fundamental group of an infinite cover of $\hat F^4$, which is a non-compact aspherical 4-dimensional object.
Therefore $H^4(A) = H^4(B)=0$, a contradiction.
\end{proof}

The proof of Corollary \ref{main:cor} is complete.

\section{Further research} \label{further:section}
There is a heuristic aspect on the results exposed in this paper which is worth noting: we did not construct deliberately a hyperbolic 5-manifold that fibers; on the contrary, we picked the simplest hyperbolic 5-manifold $M^5$ that we know, we tried to construct some fibration $f$ on it, and we were quite surprised to find one with moderate effort. This might suggest that there could be plenty of fibering hyperbolic 5-manifolds, and that our present difficulty in exhibiting more examples is only due to the intrinsically complicated combinatorics involved with any of these objects and to our lack of knowledge.
A first natural direction of research consists of finding more examples. 

\begin{quest}
How many commensurability classes of hyperbolic 5-manifolds that fiber over the circle are there? Are there any compact examples? 
\end{quest}

The technology exposed here applies \emph{as is} only to right-angled polytopes, and we know only one commensurability class of right-angled hyperbolic 5-polytopes. Another natural question is the following.

\begin{quest}
Are there hyperbolic manifolds that fiber in dimension $7,9, \ldots$ ?
\end{quest}

More generally, it is asked in \cite{BM} whether there are hyperbolic manifolds with perfect circle-valued Morse functions in all dimensions: at present we know that these exist for all $n\leq 5$. Note that in dimension $n=4$ there are both cusped and compact examples \cite{BM}. A bolder question raised in \cite{BM} is whether any hyperbolic manifold of any dimension should be finitely covered by one that has a perfect circle-valued Morse function. In odd dimensions, this is equivalent to asking whether every odd-dimensional hyperbolic manifold fibers virtually. Not a single counterexample seems known at present.

Fibering is an open condition in cohomology. Given an odd-dimensional hyperbolic manifold $M$, there is a (possibly empty) open cone $U\subset H^1(M,\matR)$ such that the classes represented by fibrations are precisely those in $U\cap H^1(M,\matZ)$. In dimension 3 we know from Thurston's work that $U$ is \emph{polytopal}, that is the cone over some open facets of a polytope. It is asked in \cite{BM} whether this holds in any dimension. It would be interesting to fully understand at least some concrete examples.

\begin{probl}
Determine $U$ for a specific fibering hyperbolic 5-manifold $M$. Is $U$ polytopal? Is it dense in $H^1(M,\matR)$? Can we find examples where $U$ is dense, and others where it is not dense?
\end{probl}

Note that this question is trivial for the Ratcliffe -- Tschantz manifold $N^5$ since $H^1(N^5,\matR) = \matR$. One might look at some cover of $N^5$ with larger first Betti number. 

Whenever a hyperbolic 3-manifold fibers over the circle, we know that the fiber is a finite-type surface and the monodromy is pseudo-Anosov. Here we were able to determine the fiber for $N^5$ in the form of a 4-dimensional ideal triangulation that Regina can handle. However, we could not determine the monodromy. The monodromy is a self-diffeomorphism of the fiber that is certainly of interest here.  

\begin{probl}
Determine the monodromy in some example. Study its dynamics. 
\end{probl}

Every fibration of a hyperbolic 5-manifold over the circle induces an abelian covering that is a geometrically infinite hyperbolic 5-manifold diffeomorphic to a product $F\times \matR$. In dimension 3 there are many beautiful theorems on such manifolds, while in dimension 5 we know absolutely nothing for the moment, except that they exist.

\begin{probl}
Consider a geometrically infinite hyperbolic 5-manifold diffeomorphic to $F\times \matR$. Is it rigid? Can we define some reasonable ending invariant? Is there any such manifold that does not cover a finite-volume one?
\end{probl}

\end{document}